\documentclass[11pt,a4paper]{amsart}
\usepackage[utf8]{inputenc}
\usepackage[T1]{fontenc}
\usepackage[UKenglish]{babel}
\usepackage[a4paper,margin=1in]{geometry}

\usepackage{amsmath,amsthm,amsfonts,amssymb}
\usepackage{mathrsfs,dsfont}
\usepackage{graphicx}
\usepackage{url}
\usepackage{verbatim}
\usepackage[dvipsnames]{xcolor}

\usepackage{import}
\usepackage{xifthen}
\usepackage{pdfpages}
\usepackage{transparent}

\usepackage{hyperref}
\usepackage{mathtools}

\usepackage{bbm}

\usepackage{marginnote}

\makeatletter
\def\namedlabel#1#2{\begingroup
	#2%
	\def\@currentlabel{#2}%
	\label{#1}\endgroup
}
\makeatother

\DeclarePairedDelimiter\norm{\lVert}{\rVert}%

\theoremstyle{plain}
\newtheorem{theorem}{Theorem}[section]
\newtheorem{corollary}[theorem]{Corollary}
\newtheorem{lemma}[theorem]{Lemma}

\theoremstyle{definition}

\newtheorem{example}[theorem]{Example}

\numberwithin{equation}{section}

\renewcommand\labelenumi{\textup{\alph{enumi})}}

\renewcommand\theenumi\labelenumi
\makeatletter\renewcommand{\p@enumii}{}\makeatother 

\renewcommand{\leq}{\leqslant}

\renewcommand{\geq}{\geqslant}
\renewcommand{\ge}{\geqslant}

\newcommand{\R}{\mathds{R}}

\newcommand{\mL}{\mathbf{L}}

\newcommand{\I}{\mathds{1}}

\begin{document}
	
	\title[Smoothing of intrinsic semigroups in non-local settings]
	{Smoothing on $L^1$ for ground state transformed semigroups in non-local settings}
	
	\date{\today}
	
		\author[M.~Baraniewicz]{Miłosz Baraniewicz}
	\address[M.~Baraniewicz]{Faculty of Pure and Applied Mathematics\\ Wroc{\l}aw University of Science and Technology\\ ul. Wybrze{\.z}e Wyspia{\'n}skiego 27, 50-370 Wroc{\l}aw, Poland}
	\email{milosz.baraniewicz@pwr.edu.pl}
	
	\author[K.~Kaleta]{Kamil Kaleta}
	\address[K.~Kaleta]{Faculty of Pure and Applied Mathematics\\ Wroc{\l}aw University of Science and Technology\\ ul. Wybrze{\.z}e Wyspia{\'n}skiego 27, 50-370 Wroc{\l}aw, Poland}
	\email{kamil.kaleta@pwr.edu.pl}

	\thanks{
	Research was supported by National Science Centre, Poland, grant no.\ 2019/35/B/ST1/02421.
}
		\begin{abstract}
We study the $L^1$-smoothing properties for a broad class of semigroups arising from the ground state transformation of Schr\"odinger semigroups with confining potentials, associated with non-local L\'evy operators, for which (asymptotic) ultracontractivity and hypercontractivity fail. Our work is inspired by Talagrand’s \emph{convolution conjecture} in the discrete cube setting, as well as by subsequent developments on the classical Ornstein--Uhlenbeck semigroup. The estimates we provide exhibit a clear dependence on the potential and the L\'evy measure defining the kinetic term operator, and they yield a description of the semigroups’ action on $L^1$ in terms of Orlicz spaces. Our framework is quite general, encompassing fractional and relativistic Laplacians as kinetic operators. The results are illustrated by numerous examples, demonstrating that the $L^1$-regularizing effects become stronger as $t \uparrow \infty$. 

	\medskip
	

\medskip

\noindent
2020 {\it MS Classification}:\ 47D08, 60G51, 47D07, 47G30, 60J35, 35S05. 
	\end{abstract}

	\maketitle
	\section{Motivation, introduction, setting and results}

	\subsection{Motivation and introduction} Let $\mu$ be the standard Gaussian measure on $\R^d$, where $d \in \left\{1,2,\ldots\right\}$, and let $\{Q_t:t \geq 0\}$ be the classical \emph{Ornstein--Uhlenbeck semigroup}, defined by
		\[
		Q_t h(x) = \int_{\R^d} h(e^{-t} x + \sqrt{1-e^{-2t}} y) \mu(dy), 
		\]
	for any admissible function $h$ on $\R^d$. Nelson's \emph{hypercontractivity} theorem \cite{Nelson73} states that for $1 < p < q < \infty$, there exists $t(p,q) >0$ such that for all $t \geq t(p,q)$, the operator $Q_t: L^p(\mu) \to L^q(\mu)$ is contractive. This implies that $Q_t$ has a regularizing effect on functions $h \in L^p(\mu)$ whenever $p>1$. Hypercontractivity is a powerful tool in the theory of operator semigroups, with numerous interesting and important connections and applications; see, e.g., \cite{DGS,Eldan-Lee,Gross} and additional references therein. 
	
	This raises a natural question:\ what happens when $p=1$? Research in this direction was inspired by Talagrand's \emph{convolution conjecture} \cite{Talagrand}. It was originally formulated in a discrete setting for the average of the heat semigroup on the hypercube but naturally extends to a much broader framework, including the Ornstein–Uhlenbeck semigroup and even the semigroups studied in the present paper. The conjecture states that, for every $t>0$ there exists a function $\psi_t :[1,\infty) \to [1,\infty)$ such that $\psi_t(u) \to \infty$ as $u \to \infty$, and for every $h \in L^1(\mu)$ with $\left\|h\right\|_{L^1(\mu)} = 1$ and $u \geq 1$, the following holds: 
  \[
	\mu\left(\left\{x \in \R^d: |Q_t h(x)| > u \right\}\right) \leq  \frac{1}{u \, \psi_t(u)}.
	\]
We have $\int Q_t h(x) \mu(dx) = \int h(x)\mu(dx)$ and, by Markov's inequality,
\[
	\mu\left(\left\{x \in \R^d: |Q_t h(x)| > u \right\}\right) \leq  \frac{1}{u} , \quad \text{whenever} \ \left\|h\right\|_{L^1(\mu)} = 1.
\]
Therefore, the function $\psi_t(u)$ can be interpreted as the rate at which the action of the operator $Q_t$ smooths the functions $h \in L^1(\mu)$. Talagrand observed that the best decay rate one can expect is $\psi_t(u) = c(t) \sqrt{\log u}$, where the constant $c(t)$ depends only on $t$. 

The first positive answear to this problem for the Ornstein--Uhlenbeck semigroup was provided by Ball, Barthe, Bednorz, Oleszkiewicz, and Wolff \cite{BBBOW}, who proved that the estimate holds true with $\psi_t(u) =  c(t,d)\sqrt{\log u}/\log\log u$, where the constant $c(t,d)$ depends on the dimension $d$. 
Later, Eldan and Lee \cite{Eldan-Lee} improved the bound to $\psi_t(u) =  c(t)\sqrt{\log u/\log\log u}$, where the constant $c(t)$ depends only on $t$. Finally, Lehec \cite{Lehec} obtained the optimal estimate with $\psi_t(u) = c \, \min\{1,t\} \sqrt{\log u}$, where $c$ is an absolute constant. We refer the reader to \cite{Eldan-Lee} for an excellent overview of this problem and methods, along with a discussion of various connections to related topics and works. See also the work of Gozlan, Madiman, Roberto, and Samson \cite{GMRS} for an interesting continuation of the research on deviation inequalities, and that of Gozlan, Li, Madiman, Roberto, and Samson \cite{GLMRS} for extensions to more general diffusions, the $M/M/\infty$ queue model, and the Laguerre semigroup in dimension one.

The aim of this paper is to explore the $L^1$-smoothing effects for a broad class of semigroups $\{Q_t:t \geq 0\}$ arising from the ground state transformation of semigroups associated with Schr\"odinger operators $H=-L+V$. Here, $L$ denotes a non-local L\'evy operator, including, as particular cases, \emph{fractional} and \emph{relativistic Laplacians}, and $V$ denotes a confining potential,\ i.e.,\ $V(x) \to \infty$ as $|x| \to \infty$. These semigroups can be seen as natural non-local analogues, or even extensions, of classical diffusion semigroups such as the Ornstein--Uhlenbeck semigroup. We note that the stationary measures $\mu$ are no longer Gaussian, as they are now given by the squares of the ground state eigenfunctions $\varphi_0$ of the Schr\"odinger operators $H$.
Hypercontractivity and ultracontractivity (corresponding to the strongest $L^1$-to-$L^{\infty}$ smoothing) of the semigroups $\{Q_t:t \geq 0\}$ have been extensively studied. Both properties are known to hold when the potential $V$ grows sufficiently fast at infinity (see, for instance, Chen and Wang \cite{bib:ChW2015,bib:ChW2016}, Kulczycki and Siudeja \cite{bib:KS}, or Kaleta, Kwaśnicki, and Lőrinczi \cite{bib:KKL2018}). In contrast, the $L^1$-smoothing properties in non-ultracontractive regime remain completely unexplored. To our knowledge, this work is the first to address this problem within the non-local L\'evy framework. We present a unified approach that enables a systematic analysis and yields a solution for a broad class of $L$'s and $V$'s.

Our main results are presented in Section \ref{sec:results}, including the upper bound in Theorem \ref{thm:main} and the lower bounds in Theorems \ref{prop:main2} and \ref{prop:main3}. While these estimates are very much in the spirit of Talagrand's convolution conjecture, the structure of the rates we obtain differs significantly from those discussed above, reflecting the sharp dependence on the kinetic term L\'evy operators and the confining potentials. Corollary \ref{cor:cor1} characterizes the range $Q_t \big(L^1(\mu)\big)$ in terms of Orlicz spaces. 

To put our results into context and demonstrate some of their direct applications, we briefly discuss an example concerning the fractional Schr\"odinger operator with a power-logarithmic potential
\[
H = (-\Delta)^a + \log^{\theta}(1+|x|),
\qquad a\in(0,1),\ \theta>0.
\]
The long-time regularizing properties of the associated ground state-transformed semigroup $\{Q_t:t\ge0\}$
exhibit a sharp dichotomy depending on the parameter $\theta$. More precisely, it is known, see \cite{bib:ChW2015,bib:KKL2018}, that 

\smallskip
\noindent
$\bullet$ if $\theta \ge 1$, then hypercontractivity and asymptotic ultracontractivity hold; that is, there exists \( t_0>0 \) such that for all \( t\ge t_0 \) the operator \( Q_t \) maps \( L^1(\mu) \) continuously into \( L^{\infty}(\mu) \);

\smallskip
\noindent
$\bullet$ if $\theta<1$, neither ultracontractivity nor hypercontractivity holds, even for large times. In this regime, the smoothing properties of the semigroup $\{Q_t:t \ge 0\}$ acting on $L^1(\mu)$, as well as the structure of the range $Q_t(L^1(\mu))$, remain unknown.

\smallskip

Our present results solve the case $\theta<1$, showing that the semigroup still exhibits nontrivial $L^1$-smoothing of a weaker type.
More precisely, for each $t>0$, the operator $Q_t$ maps $L^1(\mu)$ continuously into Orlicz spaces $\mL^{\Phi}(\mu)$ associated with Young functions
\[
\Phi(u)=|u|\exp\!\big(c\,\log^{\theta}(e+|u|)\big),
\]
provided that \( c<c_1 t \), where \( c_1=c_1(d,a,\theta)>0 \) is an explicit constant.
Moreover, this result is essentially sharp:\ if $c > c_2 t$, for some explicit $c_2>0$, then such Orlicz regularization fails. Interestingly, $\mL^{\Phi}(\mu)$ \textbf{strictly} interpolates between $L^{1}(\mu)$ and \textbf{all} spaces $L^{1+\varepsilon}(\mu)$, $\varepsilon >0$. 

If we replace the operator $H$ by its relativistic counterpart with a power-type potential
\[
H = (-\Delta+m^{1/a})^a -m +|x|^{\theta},
\qquad a\in(0,1),\ m>0,\ \theta>0,
\]
then the same sharp transition at the critical threshold $\theta=1$ occurs, and we observe a similiar Orlicz regularization for $\theta<1$. A detailed discussion of these examples is provided in Section~\ref{sec:ex}.

To state the problem and results precisely, we begin by briefly outlining the setting. Readers already familiar with the main objects of our study may simply refer to Section \ref{sec:setting} for notation and proceed directly to Section \ref{sec:results}.

	\subsection{Non-local Schr\"odinger operators and ground state-transformed semigroups} \label{sec:setting}
We consider \emph{L\'evy operators} $(L,\mathcal D(L))$ in $L^2$-setting 
which are defined as \emph{Fourier multipliers} by
\begin{align} \label{def:gen}
    \widehat{L h}(\xi) = - \Psi(\xi) \widehat{h}(\xi), \quad \xi \in \R^d, \quad h \in \mathcal D(L):=\left\{v \in L^2(dx): \Psi \, \widehat{v} \in L^2(dx) \right\},
\end{align}	
where the symbol $\Psi$ is given by
\begin{align} \label{eq:symbol}
  \Psi(\xi) =  \frac{1}{2} A \xi \cdot \xi + \int \left(1-\cos(\xi \cdot z) \right)\nu(dz), \quad \xi \in \R^d.
\end{align} 
Here, $A \in \R^{d \times d}$ is a positive semidefinite matrix, and $\nu$ is a symmetric L\'evy measure, i.e.\ a positive Radon measure on $\R^d \setminus \left\{0\right\}$ satisfying $\int \left(1\wedge |z|^2\right)\,\nu(dz) < \infty$ and $\nu(-B)=\nu(B)$. 

The symmetry of the L\'evy measure $\nu$ is equivalent to $\Psi$ being real. Consequently, the L\'evy operator $L$ is self-adjoint. 
Moreover, it may happen that $A \equiv 0$, but we always assume that 
\begin{align} \label{eq:infinite_nu}
 \nu(\R^d \setminus \left\{0\right\})=\infty,
\end{align}
which guarantees that $L$ is nonlocal and unbounded. We further assume that $\nu$ is absolutely continuous with respect to the Lebesgue measure. For simplicity, we use the same symbol for its density i.e.,\ $\nu(dx)=\nu(x)dx$.

We note that every L\'evy operator $L$ generates a strongly continuous \emph{semigroup of convolution operators} $\{P_t:t \geq 0\}$ on $L^2(dx)$, where each $P_t$ is a contraction. In particular, $L$ is the generator of a \emph{L\'evy process} with values in $\R^d$. While this stochastic background is important to the authors, it will not be used in this paper, as the proposed approach is purely analytic. Our standard references to L\'evy operators, their convolution semigroups and the corresponding L\'evy processes are monographs and lecture notes by Jacob \cite{Jacob}, Schilling \cite{bib:Sch}, B\"ottcher, Schilling and Wang \cite{BSchW}, and Sato \cite{bib:Sat}. 

Further, we consider a Schr\"odinger operator 
\[
H=-L+V, \quad \text{acting on} \ \ L^2(dx),
\] 
where the kinetic term is as above, and the potential $V$ is a locally bounded  function on $\R^d$ such that $V(x) \to \infty$ as $|x| \to \infty$, called the \emph{confining potential}. Formally, $H$ is defined in the sense of quadratic forms as a bounded below, self-adjoint operator on $L^2(dx)$ \cite[Proposition 10.22]{Schmudgen}. 

To obtain the possibly sharpest results, we impose certain regularity assumptions on the L\'evy density $\nu$ and the confining potential $V$. Specifically, we assume that there exist a strictly decreasing, continuous function $f: (0,\infty) \to (0,\infty)$, a strictly increasing, continuous function $g: [0,\infty) \to (0,\infty)$, and constants $C_1, C_2 \geq 1$ such that
\[
  C_1^{-1} f(|x|) \leq \nu(x) \leq C_1 f(|x|), \quad x\in\R^d \setminus \{0\}, 
\]
and 
\[
  C_2^{-1} g(|x|) \leq V(x) \leq C_2 g(|x|), \quad |x| \geq R_0,
\]
for some $R_0>0$. We refer to $f$ and $g$ as the \emph{profiles} of the L\'evy density $\nu$ and the potential $V$, respectively. 
We further assume that $f$ and $g$ satisfy Assumption \eqref{A}, stated and discussed in Section \ref{sec:Assumption and notation}. This assumption comes from the recent paper \cite{KSch} by Kaleta and Schilling, as the estimates established therein are the starting point of our investigations. This framework is quite general and includes Schr\"odinger operators with kinetic terms such as \emph{fractional} and \emph{relativistic Laplacians}:
\[
L=-(-\Delta)^a, \qquad L=-(-\Delta+m^{1/a})^a+m, \quad \text{where} \  a \in (0,1), \  m>0.
\]
Operators of this type and the corresponding jump processes play a fundamental role in the scientific modeling of discontinuous and non-local phenomena, especially in PDEs, mathematical physics, and the biological sciences; see, e.g., the monographs by Abatangelo, Dipierro and Valdinoci \cite{bib:AbatangeloDipierroValdinoci2025}; Lieb and Seiringer \cite{bib:LiebSeiringer2009}; Dipierro, Giacomin and Valdinoci \cite{bib:DipierroGiacominValdinoci2024}; and the references therein.
Our standard reference to Schr\"odinger operators with L\'evy kinetic terms is the monograph \cite{bib:DC} by Demuth and van Casteren.

The Schr\"odinger semigroup $\{e^{-tH}:t \geq 0\}$ consists of bounded and self-adjoint operators on $L^2(dx)$. Under Assumption \eqref{A}, $e^{-tH}$ are integral operators given by 
\[
e^{-tH} f(x) = \int_{\R^d} u_t(x,y) f(y) dy, \quad f \in L^2(dx), \ t>0.
\]
The integral kernels $u_t(\cdot,\cdot)$, $t>0$, are continuous, positive, and symmetric on $\R^d \times \R^d$. Moreover, since $V$ is a confining potential, the operator $e^{-tH}$ is compact for every $t>0$ \cite{LSW,TTT,WW}. In particular, the spectra of $H$ and $e^{-tH}$ are discrete, the lowest eigenvalue $\lambda_0:=\inf \sigma(H)$ is simple, and the corresponding eigenfunction $\varphi_0 \in L^2(dx)$ is strictly positive \cite[Theorem XIII.43]{Reed-Simon}. We refer to $\lambda_0$ and $\varphi_0$ as the \emph{ground state eigenvalue} and \emph{ground state eigenfunction}, respectively. We assume the normalization $\left\|\varphi_0\right\|_{L^2(dx)} = 1$.

\medskip

We now introduce a new (probability) measure $\mu(dx) = \varphi_0^2(x) dx$ on $\R^d$ and consider the associated weighted space $L^2(\mu)$.
The main object of our study is the semigroup of operators $\{Q_t:t \geq 0\}$, defined by
\[
Q_t h(x) := \frac{e^{\lambda_0 t}}{\varphi_0(x)} e^{-tH}(h \varphi_0)(x) = \int_{\R^d} \frac{e^{\lambda_0 t} u_t(x,y)}{\varphi_0(x)\varphi_0(y)} h(y) \mu(dy), \quad h \in L^2(\mu).
\]
Clearly, $Q_t$ acts as an integral operator with the kernel 
\[
q_t(x,y) = \frac{e^{\lambda_0 t} u_t(x,y)}{\varphi_0(x)\varphi_0(y)}.
\]	
Using the eigenequations $e^{-tH} \varphi_0(x) = \int u_t(x,y) \varphi_0(y) dy = e^{-\lambda_0 t} \varphi_0(x)$, $t>0$, and 
Jensen's inequality, we can show that this expression defines a semigroup of contractions on each space $L^p(\mu)$ for every $1 \leq p \leq \infty$. Moreover, $Q_t$ inherits the self-adjointness of $e^{-tH}$. The semigroup $\{Q_t:t \geq 0\}$ is referred to as the \emph{ground state-transformed} or \emph{intrinsic} semigroup.

The ground state transformation is a general procedure known in probability theory as the Doob transformation \cite{bib:Doob}. It allows for the construction of diffusion semigroups from classical Schrödinger semigroups \cite{Davies,Eckmann,Simon}. For instance, the Ornstein–Uhlenbeck semigroups discussed above can be derived through the ground state-transformation of the semigroup associated with the quantum harmonic oscillator, where $H = -\Delta +\frac{1}{4}|x|^2-\frac{d}{2}$ and $\varphi_0(x) = (2\pi)^{-d/4} \exp(-|x|^2/4)$. 

The hypercontractivity and ultracontractivity properties of ground state-transformed semigroups associated with Schr\"odinger operators involving nonlocal kinetic terms and confining potentials have been extensively studied. This includes, in particular, fractional and relativistic Laplacians \cite{bib:KK,bib:KS}, cylindrical fractional Laplacian \cite{bib:Kulczycki-Sztonyk}, L\'evy operators \cite{bib:KKL2018,Kaleta-Lorinczi-AoP}, and more general L\'evy-type operators \cite{bib:ChW2015,bib:ChW2016}. We note that these contractivity properties are often viewed as properties of the original Schr\"odinger semigroups—in which case they are referred to as \emph{intrinsic hypercontractivity} and \emph{intrinsic ultracontractivity}. In the non-local setting, research in this direction dates back to the pioneering works for domains by Chen and Song \cite{Chen-Song1, Chen-Song2} and Kulczycki \cite{Kul}. See also more recent important contributions by Chen, Kim and Wang \cite{ChenKimWang}, Chen and Wang \cite{ChenWangDom}, Grzywny \cite{Grzywny}, Kim and Song \cite{KimSong} and Kwa\'snicki \cite{Kwa}. For classical results concerning local Schr\"odinger operators and the historical background, we refer the reader to the seminal paper by Davies and Simon \cite{bib:DS}, the survey by Davies, Gross, and Simon \cite{DGS}, as well as the influential work of Ba\~nuelos \cite{Ban}.

\subsection{Formulation of the problem and presentation of results} \label{sec:results} For non-local operators, it is known that, in a fairly general setting—particularly under our Assumption \eqref{A}—the following statements are equivalent:\
\begin{itemize}
\item[(a)] (\emph{hypercontractivity}) For every $1 < p < q < \infty$, there exists $t(p,q)>0$ such that for every $t \geq t(p,q)$, the operator $Q_t: L^p(\mu) \to L^q(\mu)$ is bounded (here we use a general definition, see e.g.\ \cite[p.\ 371]{DGS}; in particular, we do not require that $\left\|Q_t\right\|_{L^p(\mu) \to L^q(\mu)} \leq 1$);
\smallskip
\item[(b)] (\emph{asymptotic ultracontractivity}) There exists $t_0 >0$ such that for every $t \geq t_0$, the operator $Q_t: L^1(\mu) \to L^{\infty}(\mu)$ is bounded (cf.\ \cite[Definition 2.2]{Kaleta-Lorinczi-AoP}); 
\smallskip
\item[(c)] There exist $C, R>0$ such that $V(x) \geq C |\log \nu(x)|$, whenever $|x| \geq R$ \\
           (or, equivalently, $g(|x|) \geq C |\log f(|x|)|$, $|x| \geq R$, possibly with different $C$ and $R$);
\end{itemize}
see \cite[Corollary 3.3]{bib:KKL2018}; a different, general and powerful approach can be found in \cite{bib:ChW2015}. 

The equivalence (a) $\Leftrightarrow$ (b) is somewhat surprising, as it contrasts with the classical case of diffusion semigroups, where ultracontractivity is significantly stronger than hypercontractivity--even for large times \cite{bib:DS}. For instance, the Ornstein--Uhlenbeck semigroup is hypercontractive but not asymptotically ultracontractive \cite[Example 4.4]{bib:KKL2018}. 

This equivalence also motivates the present study:\ our main goal is to understand the regularizing effects of the ground state-transformed semigroups $\{Q_t:t \geq 0\}$ on functions $h$ from $L^1(\mu)$, in situations where the condition (c), relating the behavior of $g$ and $f$, does not hold—that is, when asymptotic ultracontractivity (and hence also hypercontractivity) fails.
For technical reasons, it is convenient to assume that the map
\begin{align} \label{eq:decr_to_zero}
	r \mapsto \frac{g(r)}{|\log(f(r))|} \ \text{is eventually decreasing and} \ \lim_{r \to \infty} \frac{g(r)}{|\log(f(r))|} = 0.
\end{align}
Of course, for a given profile $f$, there exist examples of $g$ which lie strictly between the threshold in condition (c) and the decay condition \eqref{eq:decr_to_zero}. However, such cases are typically too intricate to be of practical interest in our study; illustrative examples are provided in Examples \ref{ex:ex1}-\ref{ex:ex4}.

We are now in a position to state our main results. To formulate them precisely, we first note that the constants  $K, \widetilde K>0$ appearing in the statements originate from the estimate \eqref{eq:est_noniuc}, which will be presented below. We also introduce the notation
\begin{align} \label{def:alpha}
	\alpha_t(u):= \big(f^2\big)^{-1}\left(\frac{\kappa(t)}{u}\right), \quad u  \geq \kappa(t),
\end{align}
where $\kappa(t) >0$ is a constant that will be specified in Lemma \ref{lem:alpha}. 
	
	\begin{theorem}\label{thm:main}
	Assume \eqref{A}.
If \eqref{eq:decr_to_zero} holds, then for every $t>0$  there exists a constant $C(t)>0$ such that for every $h \in L^1(\mu)$ with $\left\|h\right\|_{L^1(\mu)} = 1$ and $u \geq \kappa(t)$ we have
  \[
	\mu\left(\left\{x \in \R^d: |Q_t h(x)| > u \right\}\right) \leq  \frac{1}{u} \cdot \frac{C(t)}{w_t\big(\alpha_t(u)\big)}, 
	\]
where 
\[
w_t(r):= g^2(r) \exp(K t g(r)),
\]
with the constant $K$ that comes from the upper bound in \eqref{eq:est_noniuc}, and $\alpha_t(u)$ is defined by \eqref{def:alpha}.
  \end{theorem}

The accuracy of our upper estimate within the class of operators $L$ and the potentials $V$ described by Assumption \eqref{A} can be assessed through the following theorem. A more detailed discussion of the sharpness of these results is provided below.

\begin{theorem}\label{prop:main2}
Assume \eqref{A}. 
If \eqref{eq:decr_to_zero} holds, then  for every $t>0$  there exist constants $\widetilde C(t)>0$ and $u_0(t)>0$ such that for every $u \geq u_0(t)$ we have
  \begin{align} \label{ieq:lower_sup}
	\sup_{\genfrac{}{}{0pt}{2}{h \geq 0}{\int h d\mu = 1}} \mu\left(\left\{x \in \R^d: Q_t h(x) > u \right\}\right) \geq  \frac{1}{u} \cdot \frac{\widetilde C(t)}{\widetilde w_t\big(\beta_t(u)\big)}, 
	\end{align}
where 
\begin{align} \label{eq:tilde-w}
\widetilde w_t(r):= g^2(r) \exp(\widetilde K t g(r)),
\end{align}
with the constant $\widetilde K$ that comes from the lower bound in \eqref{eq:est_noniuc}, and 
\[
\beta_t(u) := G_t^{-1} \left(\frac{\widetilde \kappa(t)}{u}\right), \quad u \geq u_0(t),
\]
where $G_t(r)= f^2(r) \exp(\widetilde K t g(r))$ and $\widetilde \kappa(t) >0$ is a constant.
\end{theorem}
\medskip
\noindent
It is important to note that both rate functions, $w_t\big(\alpha_t(u)\big)$ and $\widetilde w_t\big(\beta_t(u)\big)$, grow more slowly than any power function, i.e.,\ for every $\delta>0$ and $ t > 0$, we have
\begin{align}\label{eq:slow_rate_intro}
\lim_{u \to \infty} \frac{w_t\big(\alpha_t(u)\big)}{u^{\delta}} = \lim_{u \to \infty} \frac{\widetilde w_t\big(\beta_t(u)\big)}{u^{\delta}} = 0;
\end{align}
see Lemma \ref{lemma:lem2}. In this connection, we also note that the main difference between the rates $w_t\big(\alpha_t(u)\big)$ and $\widetilde w_t\big(\beta_t(u)\big)$ lies in the constants $K$ and $\widetilde K$.
Indeed, although the functions $\alpha_t$ and $\beta_t$ differ, their compositions with the profile $g$ -- namely $g(\alpha_t(u))$ and $g(\beta_t(u))$ -- are typically asymptotically equivalent:\
\begin{align} \label{eq:res_rate}
\lim_{u \to \infty} \frac{g(\alpha_t(u))}{g(\beta_t(u))} = 1,
\end{align} 
see Lemma \ref{lemma:lem2} for formal statement and illustration in Examples \ref{ex:ex1}-\ref{ex:ex4}. In particular, the specific values of the constants $\kappa(t)$ and $\widetilde \kappa(t)$ do not affect the asymptotic behavior of $g(\alpha_t(u))$ and $g(\beta_t(u))$. 

For potentials $V$ satisfying \eqref{eq:decr_to_zero} and growing sufficiently fast at infinity, Theorem \ref{thm:main} yields an integrable upper bound. In this case, it provides an upper estimate for the range $Q_t(L^1(\mu))$, see Corollary \ref{cor:cor1} (a) below; however, the sharpness of this inclusion remains unclear  as Theorem \ref{prop:main2} is insufficient to verify it.  Our final result constructs a function $h \in L^1(\mu)$ and establishes a lower bound for $\mu\left(\left\{x \in \R^d: Q_t h(x) > u \right\}\right)$, yielding additional information about the range of $Q_t(L^1(\mu))$.
 
\begin{theorem}\label{prop:main3} 
Assume \eqref{A} and \eqref{eq:decr_to_zero}. Let $\eta, \sigma:[0,\infty) \to (0,\infty)$ be increasing and continuous functions such that $1/\eta \in L^1((1,\infty),dx)$, the map
\begin{align} \label{eq:eta_decr_to_zero}
	r \mapsto  \frac{(d-1)\log r + \log \eta(r)}{|\log f(r)|} \ \text{is eventually decreasing and} \ \lim_{r \to \infty} \frac{(d-1)\log r + \log \eta(r)}{|\log f(r)|} < 2,
\end{align}
$\sigma(r) < r$, for $r \geq 1$, and 
\begin{equation} \label{theta-bling}
	\sup_{r \geq 1} \frac{f(\sigma(r))}{f(r)} < \infty.
\end{equation}
Then the function $h:\R^d \to (0,\infty)$, defined by 
\[h(x):= \frac{g^2(|x|)}{\eta(|x|)|x|^{d-1}  f_1^2(|x|)},\] 
belongs to $L^1(\mu)$ and  for every $t>0$  there exist constant $\widetilde C(t)>0$ and $u_0(t)>0$ such that for every $u \geq u_0(t)$ we have
  \begin{align} \label{ieq:lower_h}
	\mu\left(\left\{x \in \R^d: Q_t h(x) > u \right\}\right) \geq  \frac{1}{u} \cdot \frac{\widetilde C(t) }{v_t\big(\gamma_t(u)\big) },
	\end{align}
	where 
    \[
v_t(r):= \widetilde w_t(r) \frac{\eta\big(r\big)}{r - \sigma(r)} 
       = g^2(r) \exp\big(\widetilde K t g(r)+\log \eta(r) - \log(r - \sigma(r) )\big)
    \]
and 
	\[
  \gamma_t(u) := H_t^{-1} \left(\frac{\widetilde \kappa(t)}{u}\right), \quad u \geq u_0(t),
  \]
  where $H_t(r) = f^2(r) r^{d-1}\eta(r) \exp(\widetilde K t g(r))$, with the constant $\widetilde K$ coming from the lower bound in \eqref{eq:est_noniuc}, and $\widetilde \kappa(t) >0$ is a constant.

\end{theorem}	

\medskip
\noindent
The statement of the above theorem is necessarily technical. It assumes the existence of two auxiliary functions $\eta$ and $\sigma$ that satisfy \eqref{eq:eta_decr_to_zero} and \eqref{theta-bling}. Although such functions always exist, the key difficulty is to choose them so that the lower bound \eqref{ieq:lower_h} is as close as possible to that in Theorem \ref{thm:main}. For instance, by \eqref{eq:f_comp} one may always take $\sigma(r) = r-1$, which improves to $\sigma(r) = \frac{r}{2}$ when $f$ has the doubling property. Admissible choices for $\eta$ include $\eta(r) \approx r \log^2r$ or $\eta(r) \approx r \log r \log^2 \log r$, etc., depending on the growth of the potential. Section \ref{sec:ex} illustrates this for two extremal cases, where the profile $f$ decays polynomially and exponentially at infinity. A variant of \eqref{eq:slow_rate_intro}, \eqref{eq:res_rate} for $v_t\big(\gamma_t(u)\big)$ and its exponent will be provided in Lemma \ref{lemma:lem3}.

The estimates in Theorems~\ref{thm:main} and~\ref{prop:main3} allow us to describe the range $Q_t(L^1(\mu))$ in terms of Orlicz spaces, provided that the potential $V$ grows sufficiently fast at infinity. Recall that a function $\Phi: \R \to [0,\infty)$ is a \emph{Young function} if it is convex, even, and satisfies $\Phi(0) = 0$ and $\lim_{u \to \infty}\Phi(u) = \infty$, see \cite{Rao-Ren?}. The corresponding \emph{Orlicz space} is defined as
\[
\mL^{\Phi}(\mu) = \left\{h:\R^d \to \R \ \ \text{measurable and} \ \int_{\R^d} \Phi(|h(x)|/\lambda) \mu(dx) < \infty \ \ \text{for some} \ \lambda>0 \right\};
\]
$\mL^{\Phi}(\mu)$ is a linear space, and the norm
\[
\left\|h\right\|_{\mL^{\Phi}(\mu)} = \inf \left\{\lambda>0: \int_{\R^d} \Phi(|h(x)|/\lambda) \mu(dx)  \leq 1 \right\},
\]
known as the \emph{Luxemburg norm}, endows $\mL^{\Phi}(\mu)$ with the structure of a Banach space. In particular, if $\Phi(u)=|u|^q$ with $q \in [1,\infty)$, then $\mL^{\Phi}(\mu) = L^q(\mu)$.

\begin{corollary}\label{cor:cor1}
Assume \eqref{A} and \eqref{eq:decr_to_zero}. Let $\Phi$ be a Young function and let $ t >0$. The following assertions hold. 
\begin{itemize}
\item[(a)] If there is $\lambda>0$ such that
\[
\int_{\frac{\kappa(t)}{\lambda}}^{\infty}  \frac{1}{u}\frac{\Phi^{\prime} (u)}{w_t(\alpha_t(\lambda u))} du < \infty,
\]
then $Q_t$ maps $L^1(\mu)$ continuously into $\mL^{\Phi}(\mu)$.

\item[(b)] Let $\eta, \sigma:[0,\infty) \to (0,\infty)$ be increasing and continuous functions such that $1/\eta \in L^1((1,\infty),dx)$, $\sigma(r) < r$, for $r \geq 1$, and \eqref{eq:eta_decr_to_zero}, \eqref{theta-bling} hold. If for every $\lambda>0$,
\[
\int_{\frac{u_0(t)}{\lambda}}^{\infty}  \frac{1}{u}\frac{\Phi^{\prime} (u)}{v_t(\gamma_t(\lambda u))} du = \infty,
\]
then $Q_t \big(L^1(\mu)\big) \not \subset \mL^{\Phi}(\mu)$. 
\end{itemize}
\end{corollary}
We refer the reader to Section~\ref{sec:ex} for concrete examples of Orlicz spaces arising in the context of fractional and relativistic Schr\"odinger operators with various confining potentials.


\subsection{Dependence on parameters and sharpness of the results} \label{sec:sharp} One of the main concerns in the study of the classical Ornstein--Uhlenbeck semigroup was whether the constant could be chosen independently of the dimension $d$. We emphasize that our setting is fundamentally different, as we do not focus on a specific semigroup. Instead, we allow the L\'evy measure $\nu$ (or the L\'evy operator $L$) and the confining potential $V$ to vary within a fairly broad class. The main objective of this paper is to understand how the functions $w_t$ and $\alpha_t$ (or $\widetilde w_t$ and $\beta_t$) depend on the decay or growth properties of $\nu$ and $V$. In particular, we allow the constants to depend on $\nu$, $V$; they may also implicitly depend on the dimension $d$ through $\nu$. For simplicity, this dependence is suppressed in the notation. At the level of generality considered here, achieving independence of these constants from the initial data appears to be out of reach.
	
We note that even for very specific choices of $L$ and $V$, there are no explicit formulas for (the heat kernels of) the Schrödinger semigroups $\{e^{-tH}:t \geq 0\}$ or the corresponding intrinsic semigroups $\{Q_t:t \geq 0\}$. In this respect, our present setting also differs essentially from the classical Ornstein--Uhlenbeck case. As a consequence, we have to rely on the approximate results available in the literature. Our findings rely on two-sided heat kernel estimates, established only recently in \cite[Theorem 1.1]{KSch} (see also Theorem 4.6 therein). Combined with  the ground state estimates \eqref{eq:gs}, this result implies that for every $T>0$ there exist $\rho>1$, a constant $C=C(T)>0$, and constants $K, \widetilde K>0$ (independent of $T$)  such that for every $x,y \in \R^d$ and all $t \geq T$, we have
  \begin{align}\label{eq:est_noniuc}
	C^{-1} \max\left\{1, e^{\lambda_0 t}\Gamma(\widetilde K t, x, y)\right\} \leq q_t(x,y) \leq C \max\left\{1, e^{\lambda_0 t}\Gamma(K t, x, y)\right\},
  \end{align}
  where
	\begin{align}\label{eq:F_def}
		\Gamma (\tau, x, y) := \frac{\I_{\left\{|x|, |y| > \rho \right\}}}{f_1(|x|) f_1(|y|)} \int_{\rho-1 < |z| < |x| \vee |y|} f_1(|x-z|)f_1(|z-y|) \frac{dz}{\exp(\tau g(|z|))}.
	\end{align} 
We note that the original statement of \cite[Theorem 1.1]{KSch} uses $T=30t_b$, with $t_b>0$ from Assumption \eqref{A}. Here, we allow any $t_b>0$ in Assumption \eqref{A}, so the estimates \eqref{eq:est_noniuc} hold for all $t \geq T$, for any fixed $T>0$. 

 We have made a considerable effort to find an argument strong enough to preserve the same constants $K$, $\widetilde K$ in exponential terms. In this sense, our results in Theorems \ref{thm:main} and \ref{prop:main2} are just as sharp as the estimates in \eqref{eq:est_noniuc}, without loss of information. This is reflected in our proofs:\ we show that if \eqref{eq:est_noniuc} holds for some  constants $K, \widetilde K>0$, then our results remain valid for the same values of $K$ and $\widetilde K$. Although the constants $K, \widetilde K$ found in \cite{KSch} are not optimal (in our notation $K = \widetilde K^{-1} = (4C_2 C_6^2)^{-1}$, where $C_2$ and $C_6$ are from Assumption \eqref{A}), we believe that \eqref{eq:est_noniuc} still captures the essential behavior of the semigroups under consideration. These constants could potentially be improved in future work. 
 
\subsection{Notation} For $x\in \R^d$ and $r>0$, we denote $B_r(x) := \{z\in\R^d : |z-x| <r\}$. We set $L^p(\mu) := L^p(\R^d,\mu)$ and write $L^p(dx)$ when $\mu$ is the Lebesgue measure. We write $a \wedge b:= \min\left\{a,b\right\}$, $a \vee b:= \max\left\{a,b\right\}$, and, for a real-valued function $h$, we define $h_1:=h \wedge 1$. In addition, $\alpha(u) \approx \beta(u)$ means that $\lim_{u \to \infty}\alpha(u)/\beta(u) = 1$. The inverse of a function $f$ is denoted by $f^{-1}$.

	\section{Assumption \eqref{A} and its consequences}\label{sec:Assumption and notation}

We now introduce the main assumption, which concerns the regularity of the kinetic term operator $-L$, its associated convolution semigroup $\{P_t:t \geq 0\}$, and the confining potential $V$. Since the L\'evy density $\nu(x)$ exists and condition \eqref{eq:infinite_nu} holds, the operators $P_t$, $t>0$, are integral operators \cite[Theorem 27.7]{bib:Sat}. That is, for every $t > 0$, there exists a positive density $p_t(\cdot)$ such that 
\[
P_t f(x) = \int_{\R^d} p_t(y-x) f(y) dx.
\]
Further information on the existence and regularity of the (transition) densities for convolution semigroups can be found in \cite{KSch}.
\medskip

\noindent
\textbf{Assumption (\namedlabel{A}{A})} \\
The L\'evy density $\nu(x)$ and the confining potential $V(x)$ have \emph{radial profiles}:\ there exist a strictly decreasing, continuous function $f: (0,\infty) \to (0,\infty)$, a strictly increasing, continuous function $g: [0,\infty) \to (0,\infty)$, and constants $C_1, C_2 \geq 1$ such that
\[
  C_1^{-1} f(|x|) \leq \nu(x) \leq C_1 f(|x|), \quad x\in\R^d \setminus \{0\},
\]
and 
\[
  C_2^{-1} g(|x|) \leq V(x) \leq C_2 g(|x|), \quad |x| \geq R_0,
\]
for some $R_0>0$. 
Moreover, the following conditions hold. 
	\begin{itemize}\itemsep=5pt
\item[\bfseries(\namedlabel{A1}{A1})] 
	There exists a constant $C_3 > 0$ such that
		\begin{equation*}
			\int_{\genfrac{}{}{0pt}{2}{|x-y|>1}{|y|>1}} f(|x-y|)f(|y|) dy \leq C_3 f(|x|), \quad |x| \geq 1;
		\end{equation*}  
		
\item[\bfseries(\namedlabel{A2}{A2})] 
	The function $(t,x) \mapsto p_t(x)$ is continuous on $(0,\infty) \times \R^d$ and  for every $t_b>0$ there exist constants  $C_4, C_5 >0$ such that
		\begin{equation*}
			p_t(x) \leq C_4 ( [e^{C_5t} f(|x|)]\wedge 1), \quad x\in\R^d\setminus \{0\}, \ t \geq t_b,
		\end{equation*}
		and, for every $r \in (0,1]$,
		\begin{equation*}
			\sup\limits_{t \in (0,t_b]} 	\sup\limits_{r \leq |x| \leq 2} p_t(x)< \infty;
		\end{equation*}

\item[\bfseries(\namedlabel{A3}{A3})] 
There exists a constant $C_6 \geq 1$ such that $g(r+1) \leq C_6 g(r)$ for $ r \geq R_0$. 
\end{itemize}

\bigskip
	
Assumption \eqref{A} provides a regularity framework under which the estimates in \eqref{eq:est_noniuc}, established in \cite[Theorem 1.1]{KSch}, can be applied. These conditions have been extensively discussed in the cited work; rather than repeating that discussion here, we will only mention that \eqref{A1}-\eqref{A2} define a very broad class of L\'evy operators/processes. This class includes those with power-law decaying L\'evy measures, such as the fractional Laplacian/isotropic stable process, as well as those with exponentially decaying L\'evy measures, such as the relativistic stable operator/process, and also their mixtures with the standard Laplacian/Brownian motion. Examples are discussed in Section~\ref{sec:ex}. Condition \eqref{A3} characterizes potentials with at most exponential growth at infinity. For example, $V(x)=\log(1+|x|)^{\gamma}$ and $V(x)=|x|^{\gamma}$, with $\gamma>0$, are in this class. We refer the reader to the respective assumptions (A1)–(A3) in the cited paper, including the discussion and the sufficient conditions under which they are satisfied.  We note on passing that our Assumption \eqref{A} is formally slightly more restrictive than in that paper. To obtain the sharpest results, we assume both profiles $f$ and $g$ are continuous and strictly monotone, and that (A2) holds for every fixed $t_0>0$. 

Due to the fact that $\int_{|y|>1} \nu(y) dy < \infty$, we have $\lim\limits_{r \to \infty} f(r) = 0$.
Moreover, by \eqref{A1} there exists a constant $C_7 \geq 1$ such that
		\begin{align}\label{eq:f_comp}
		f(r) \leq C_7 f(r+1), \quad r \geq 1.
		\end{align}
	From \cite[Lemma 3.1]{KSch} we know that the convolution condition \eqref{A1} can be restated as follows:\ there exists a constant $C_8 >0$ such that
	\begin{equation}\label{eq:DJP}
		\int_{\R^d} f_1(|x-z|) f_1(|z-y|) dz \leq C_8 f_1(|x-y|), \quad x,y \in \R^d,
	\end{equation}
and that there is $C_9>0$ such that
\begin{equation}\label{eq:DJP_point}
		f_1(|x|)f_1(|x-y|) \leq C_9 f_1(|y|), \quad x,y \in \R^d.
\end{equation}
We note that the ground state eigenfunction $\varphi_0$ admits a version that is continuous and bounded on $\R^d$; see, e.g., \cite[dicussion on p.\ 12]{KSch}. We always work with this version. Combining this with the fact that $\varphi_0$ is strictly positive, and using the sharp two-sided estimate established in \cite[Corollary 2.2]{Kaleta-Lorinczi-AoP}, we find that there exists a constant $C_{10} \geq 1$ such that
	\begin{equation}\label{eq:gs}
		\frac{1}{C_{10}}\frac{f_1(|x|)}{g(|x|)}	\leq \varphi_0(x) \leq C_{10} \frac{f_1(|x|)}{g(|x|)}, \quad x \in \R^d. 
	\end{equation}
	
	\section{Proof of Theorem \ref{thm:main}}\label{sec:Upper_bound}
	
Throughout this section, we assume that \eqref{A} and \eqref{eq:decr_to_zero} hold.	The proof of the main result is based on two lemmas. The starting point of our analysis is the following upper bound for the kernel $q_t(x,y)$:\  for every $T>0$ there exist $\rho>1$, a constant $c_1=c_1(T)>0$ and a constant $K>0$ (independent of $T$)  such that, for every $x,y \in \R^d$ and all $t \geq T$,
  \begin{align}\label{eq:est_noniuc_proof}
	q_t(x,y) \leq c_1 \left(1 + e^{\lambda_0 t}\frac{\I_{\left\{|x|, |y| > \rho \right\}}}{f_1(|x|) f_1(|y|)} \int_{\rho-1 < |z| < |x| \vee |y|} f_1(|x-z|)f_1(|z-y|) e^{-K t g(|z|)} dz \right),
  \end{align}
see \eqref{eq:est_noniuc}. As mentioned in the Introduction, this bound was established in \cite{KSch} for a suitable choice of $K$. However, we emphasize that our argument in this section is general:\ we show that if inequality \eqref{eq:est_noniuc_proof} holds for some $K>0$, then Theorem \ref{thm:main} remains valid for the same value of $K$.

	The first lemma shows that the function $\alpha_t$ that appears in our upper estimate should take the form
		\[
		\alpha_t(u)= \big(f^2\big)^{-1}\left(\frac{\kappa(t)}{u}\right), \quad u \geq \kappa(t),
	\]
	for some constant $\kappa(t)>0$. Clearly, $\big(f^2\big)^{-1} = \big(f_1^2\big)^{-1}$ on $(0,1]$. 
	
	\begin{lemma} \label{lem:alpha}
		 For every $t>0$  there exists a constant $\kappa(t) >0$ such that for every nonnegative $h \in L^1(\mu)$ with $\left\|h\right\|_{L^1(\mu)} = 1$ we have
		\begin{equation*}
			Q_t h(x) \leq  \frac{\kappa(t)}{f_1^2(|x|)}, \quad x \in \R^d. 
		\end{equation*}
	In particular, we have the inclusion 
	\begin{equation}\label{ieq:incl}
		\{x \in \R^d: Q_t h(x)  \geq u  \} \subset \Big\{x \in \R^d: |x| \geq \alpha_t(u) \Big\}, \quad u \geq \kappa(t). 
	\end{equation}	
	\end{lemma}
	
	\begin{proof}
		Let $h \in L^1(\mu)$, $h \geq 0$ be such that $\int h d\mu = 1$ and let $ t >0$. Suppose first that $|x| \leq \rho $. Then, by \eqref{eq:est_noniuc_proof}, 
		\[
		Q_t h(x) = \int_{\R^d} q_t(x,y) h(y) \mu(dy) \leq c_1 \int_{\R^d} h(y) \mu(dy) = c_1 \leq  \frac{c_1}{f_1^2(|x|)}.
		\]
	When $|x| > \rho$, then \eqref{eq:est_noniuc_proof} yields
	\begin{align*}
		& Q_t h(x) \\ 
		    & \leq c_1\left(1+ \int_{|y| > \rho}\left(\frac{e^{\lambda_0 t}}{f_1(|x|) f_1(|y|)} \int_{\rho-1 < |z| < |x| \vee |y|} f_1(|x-z|)f_1(|z-y|) e^{-K t g(|z|)} dz \right) h(y) \mu(dy)\right).
	\end{align*}
	Using the fact that $e^{-K t g(|z|)} \leq 1$ and applying \eqref{A1} in the form \eqref{eq:DJP} to the inner integral on the right hand side, we get
	\begin{align*}
			Q_t h(x) \leq  c_1\left(1+ C_8 e^{\lambda_0 t} \int_{|y| > \rho} \frac{f_1(|x-y|)}{f_1(|x|) f_1(|y|)} h(y) \mu(dy)\right).
	\end{align*}
Finally, by \eqref{eq:DJP_point}, we obtain 
	\begin{equation*}
		Q_t h(x) \leq  c_1\left(1+ C_8 C_9 \frac{e^{\lambda_0 t}}{f_1^2(|x|)} \int_{\R^d} h(y) \mu(dy) \right) \leq \frac{\kappa(t)}{f_1^2(|x|)},
	\end{equation*}
with $\kappa(t) = c_1(1+ C_8 C_9 e^{\lambda_0 t})$. 

The claimed inclusion then follows from the implication
\[
Q_t h(x) \geq u \ \Longrightarrow \ f_1^2(|x|) \leq  \frac{\kappa(t)}{u}.
\]
This completes the proof. 
	\end{proof}

The next lemma is the main step of the proof. 

	\begin{lemma}\label{lem:sup_estimates}
		 For every $t>0$  there exists a constant $C(t)>0$ such that
	\begin{equation*}
	 \norm{Q_t \I_{B_r(0)^c}}_{L^\infty(\mu)} \leq \frac{C(t)}{g^2(r)} e^{- K t g(r)}, \quad r >0.
	\end{equation*}
	\end{lemma}
	
	\begin{proof}
		Let $ t >0$. We need some preparation. Note that for any parameters $a, b>0$
	\begin{align} \label{eq:a_b_event_decr}
	 \text{the function} \ r \mapsto f_1^a(r) e^{b t g(r)} \ \text{is eventually strictly decreasing {(to 0)} on }  (0,\infty).  
\end{align}
Indeed,
	\begin{equation*}
		f_1^a(r) e^{b t g(r)} = \exp\big(a \log(f_1(r)) +  b t g(r)\big) = \exp\left(-|\log(f_1(r))|\Big(a - \frac{ b  t g(r)}{|\log(f_1(r))|}\Big)\right) , 
	\end{equation*}
	the map $r \mapsto |\log(f_1(r))|$ eventually strictly increases to $\infty$ by \eqref{A}, and $r \mapsto a - \frac{ b  t g(r)}{|\log(f_1(r))|}$ eventually increases to $a$ by \eqref{eq:decr_to_zero}. 
	
	By \eqref{eq:a_b_event_decr},
	\begin{align} \label{eq:event_decr}
	\text{there exists} \ r_0(t) \geq \rho \ \text{such that the map} \ r \mapsto f_1(r) e^{K t g(r)} \ \text{is decreasing (to 0) on} \ [r_0(t),\infty).  
\end{align}	
 In particular, we may assume that $r_0(t)$ is large enough so that
\begin{align} \label{eq:f_1_by_exp}
	f_1(r) \leq e^{-Ktg(r)}, \quad r \geq r_0(t).
\end{align}	

Observe that when $0 < r \leq r_0(t)$, then we simply have
\[
	 \norm{Q_t \I_{B_r(0)^c}}_{L^\infty(\mu)} \leq 1 \leq \left({e^{K t g(r_0(t))}}{g^2(r_0(t))} \right)\frac{1}{g^2(r)} e^{- K t g(r)} =: \frac{C(t)}{g^2(r)} e^{- K t g(r)}.
\]
Therefore, for the rest of the proof, we assume that $r > r_0(t)$. 	
The heat kernel upper bound \eqref{eq:est_noniuc_proof} leads to the following estimates:\ 
		if $|x| \leq \rho$, then
		\[
		0 \leq Q_t \I_{B_r(0)^c}(x) = \int_{B_r(0)^c} q_t(x,y) \mu(dy) \leq c_1 \int_{B_r(0)^c} \varphi_0^2(y) dy =: c_1\textup{I}_1,
		\]
	 and if $|x| > \rho$, then
	\begin{align*}
		0 \leq Q_t \I_{B_r(0)^c}(x) & \leq c_1 \int_{|y| \geq r} \varphi_0^2(y) dy \\
		         & + c_1 \int_{|y| \geq r} \left(\frac{e^{\lambda_0 t}}{f_1(|x|) f_1(|y|)} \int_{\rho-1 < |z| \leq r_0(t)} f_1(|x-z|)f_1(|z-y|) e^{-K t g(|z|)} dz \right) \mu(dy) \\
		         & + c_1 \int_{|y| \geq r} \left(\frac{e^{\lambda_0 t}}{f_1(|x|) f_1(|y|)} \int_{r_0(t) < |z| \leq r} f_1(|x-z|)f_1(|z-y|) e^{-K t g(|z|)} dz \right) \mu(dy) \\
						 & + c_1 \int_{|y| \geq r} \left(\frac{e^{\lambda_0 t}}{f_1(|x|) f_1(|y|)} \int_{r < |z| < |x| \vee |y|} f_1(|x-z|)f_1(|z-y|) e^{-K t g(|z|)} dz \right) \mu(dy) \\
						                    & =: c_1 (\textup{I}_1+\textup{I}_2+\textup{I}_3+\textup{I}_4).
	\end{align*}
Each of the integrals $\textup{I}_i$, $i=1,\ldots,4$, requires a different approach. We start with $\textup{I}_4$. 
Since the map $s \mapsto e^{-K t g(s)}$ is decreasing, we have
		\begin{align*}
		\textup{I}_4 & \leq e^{-K t g(r)} \int_{|y| \geq r} \left(\frac{e^{\lambda_0 t}}{f_1(|x|) f_1(|y|)} \int_{r < |z| < |x| \vee |y|} f_1(|x-z|)f_1(|z-y|) dz \right) \mu(dy)\\
		   & \leq c_2 e^{\lambda_0 t} e^{-K t g(r)} \int_{|y| \geq r} \frac{f_1(|x-y|)}{f_1(|x|) f_1(|y|)}  \mu(dy),
		\end{align*}
		where in the last inequality we used \eqref{eq:DJP}. By the fact that $\mu(dy) = \varphi_0^2(y) dy$, the upper estimate in \eqref{eq:gs} and one more use of \eqref{eq:DJP}, we obtain that
		\begin{align*}
		\textup{I}_4 & \leq c_3 e^{\lambda_0 t} e^{-K t g(r)} \int_{|y| \geq r} \frac{f_1(|x-y|)f_1(|y|)}{f_1(|x|) g^2(|y|) }  dy \\
		    & \leq c_3 e^{\lambda_0 t} \frac{e^{-K t g(r)}}{g^2(r)} \int_{\R^d} \frac{f_1(|x-y|)f_1(|y|)}{f_1(|x|) }  dy \\
				& \leq c_4 e^{\lambda_0 t} \frac{e^{-K t g(r)}}{g^2(r)}.
		\end{align*}
In order to estimate $\textup{I}_3$ we first write
		\begin{align*}
			\textup{I}_3 = \int_{|y| \geq r} \left(\frac{e^{\lambda_0 t}}{f_1(|x|) f_1(|y|)} \int_{r_0(t) < |z| \leq r} f_1(|x-z|)f_1(|z|)f_1(|z-y|) \frac{1}{f_1(|z|) e^{K t g(|z|)}} dz \right) \varphi_0^2(y) dy
		\end{align*}	
and since $|z| \leq r \leq |y|$	on the domain of integration, \eqref{eq:event_decr} implies that		
			\begin{align*}
			\textup{I}_3 \leq \int_{|y| \geq r} \left(\frac{e^{\lambda_0 t}}{f_1(|x|) f_1(|y|)}\frac{1}{f_1(|y|) e^{K t g(|y|)}} \int_{r_0(t) < |z| \leq r} f_1(|x-z|)f_1(|z|)f_1(|z-y|) dz \right) \varphi_0^2(y) dy.
		\end{align*}	
Further, we have	
\begin{align*}
			\textup{I}_3 & \leq c_5 e^{\lambda_0 t} \int_{|y| \geq r} \frac{e^{-K t g(|y|)}}{f_1(|x|)g^2(|y|)}  \int_{r_0(t) < |z| \leq r} f_1(|x-z|)f_1(|z|)f_1(|z-y|) dz dy \\
					& \leq c_5 e^{\lambda_0 t} \frac{e^{-K t g(r)}}{g^2(r)} \int_{\R^d} \frac{f_1(|x-z|)f_1(|z|)}{f_1(|x|)} \left(\int_{\R^d}f_1(|z-y|) dy \right) dz \\
					& \leq c_6 e^{\lambda_0 t} \frac{e^{-K t g(r)}}{g^2(r)},
		\end{align*}
	where the first inequality follows from the upper bound in \eqref{eq:gs}, the second from the monotonicity of the profile $g$ and the Tonelli theorem, and the third is a consequence of the fact that $\int_{\R^d}f_1(|y|) dy < \infty$ and \eqref{eq:DJP}. 

We now turn to estimating of $\textup{I}_2$. We first rearrange the integral, use the monotonicity of the profile $f_1$ and the inequality $e^{-K t g(|z|)} \leq 1$:\ 
\begin{align*}
\textup{I}_2 & = e^{\lambda_0 t}\int_{|y| \geq r}  \int_{\rho-1 < |z| \leq r_0(t)} \frac{1}{f_1^2(|z|)}\frac{f_1(|x-z|) f_1(|z|)}{f_1(|x|)} \frac{f_1(|z|) f_1(|z-y|)}{f_1(|y|)} e^{-K t g(|z|)} dz  \varphi_0^2(y) dy \\
    & \leq \frac{e^{\lambda_0 t}}{f_1^2(r_0(t))} \int_{|y| \geq r}  \int_{\rho-1 < |z| \leq r_0(t)} \frac{f_1(|x-z|) f_1(|z|)}{f_1(|x|)} \frac{f_1(|z|) f_1(|z-y|)}{f_1(|y|)} dz  \varphi_0^2(y) dy.
\end{align*}
Then we use \eqref{eq:DJP_point} to show that $\frac{f_1(|x-z|) f_1(|z|)}{f_1(|x|)} \leq c_7$ and further apply 
\eqref{eq:DJP}  to get
\[
\sup_{|y| \geq r} \int_{\rho-1 < |z| \leq r_0(t)} \frac{f_1(|z|) f_1(|z-y|)}{f_1(|y|)} dz \leq \sup_{y \in \R^d} \int_{\R^d} \frac{f_1(|z|) f_1(|z-y|)}{f_1(|y|)} dz < \infty.
\]
This yields 
\[
\textup{I}_2 \leq c_8 \frac{e^{\lambda_0 t}}{f_1^2(r_0(t))} \int_{|y| \geq r}  \varphi_0^2(y) dy.
\]
In particular, 
\[
\textup{I}_1 + \textup{I}_2 \leq \left(1 + c_8 \frac{e^{\lambda_0 t}}{f_1^2(r_0(t))} \right) \int_{|y| \geq r} \varphi_0^2(y) dy.
\]
The upper estimate in \eqref{eq:gs},  monotonicity of the profiles and \eqref{eq:f_1_by_exp} imply
\[
\varphi_0^2(y) \leq c_9 \frac{f^2_1(|y|)}{g^2(|y|)} \leq c_9 \frac{f_1(r)}{g^2(r)} f_1(|y|) \leq c_{10} \frac{e^{-K t g(r)}}{g^2(r)} f_1(|y|), \quad |y| \geq r,
\]
which leads to the final bound 
\[
\textup{I}_1 + \textup{I}_2 \leq c_{10} \left(1 + c_8 \frac{e^{\lambda_0 t}}{f_1^2(r_0(t))} \right) \frac{e^{-K t g(r)}}{g^2(r)} \int_{\R^d} f_1(|y|) dy \leq c_{11} \left(1 + c_8 \frac{e^{\lambda_0 t}}{f_1^2(r_0(t))} \right) \frac{e^{-K t g(r)}}{g^2(r)}. 
\]
By combining the estimates for all the integrals $\textup{I}_i$, $i=1,\ldots,4$, we arrive at the conclusion.
	\end{proof}

   \begin{proof}[Proof of Theorem \ref{thm:main}] 
	Let $ t > 0$ and let $\alpha_t$ be the function defined in \eqref{def:alpha}. By Lemma \ref{lem:alpha} and Markov's inequality, we have, for every $h \in L^1(\mu)$ with $\left\|h\right\|_{L^1(\mu)} = 1$ and $u \geq \kappa(t)$,
		\begin{align*} 
		\mu(\{x \in \R^d: |Q_t h(x)| > u \}) & \leq \mu(\{x \in \R^d: Q_t |h|(x) > u \})  \\& = \mu\big(\{x \in \R^d: Q_t |h|(x) > u \} \cap B_{\alpha_t(u)}(0)^c\big) \\
			& \leq \frac{1}{u}\int_{\R^d} Q_t |h|(x) \I_{B_{\alpha_t(u)}(0)^c}(x)\mu (dx).
		\end{align*}
		Further, it follows from Tonelli's theorem and the symmetry of the kernels $q_t(x,y)$ that
        \begin{align*}
\int_{\R^d} Q_t |h|(x) \I_{B_{\alpha_t(u)}(0)^c}(x)\mu (dx) = \int_{\R^d} |h(y)| Q_t\I_{B_{\alpha_t(u)}(0)^c}(y)\mu (dy)
		\end{align*}
		and, consequently, by Lemma \ref{lem:sup_estimates}, we finally obtain the estimate
		\begin{align*}
			\mu(\{x \in \R^d: |Q_t h(x)| > u \}) & \leq \frac{1}{u} \norm{Q_t \I_{B_{\alpha_t(u)}(0)^c}}_{L^{\infty}(\mu)} \leq \frac{1}{u} \frac{C(t)}{g^2(\alpha_t(u))} e^{- K t g(\alpha_t(u))}.
		\end{align*}
		This completes the proof. 
	\end{proof}     
	\section{Proof of Theorem \ref{prop:main2} and Lemma \ref{lemma:lem2}}\label{sec:lower_bound}
	
We will establish the following general implication:\ if  for every $T>0$ there exists $\rho>1$, a constant $c_1=c_1(T)>0$ and a constant $\widetilde K>0$ (independent of $T$)  such that, for every $x,y \in \R^d$ and all $t \geq T$,  
  \begin{align}\label{eq:est_noniuc_proof_low}
	q_t(x,y) \geq c_1 \left(1 + e^{\lambda_0 t}\frac{\I_{\left\{|x|, |y| > \rho \right\}}}{f_1(|x|) f_1(|y|)} \int_{\rho-1 < |z| < |x| \vee |y|} f_1(|x-z|)f_1(|z-y|) e^{-\widetilde K t g(|z|)} dz \right),
  \end{align}
 then the assertion of Theorem \ref{prop:main2} holds.  
	
\begin{proof}[Proof of Theorem \ref{prop:main2}]
Assume \eqref{A} and \eqref{eq:decr_to_zero}, and fix $ t>0$. Let $h_{\varepsilon,y}(\cdot) = q_{\varepsilon t}(\cdot,y)$, $\varepsilon >0$, $y \in \R^d$. Clearly, 
\[
h_{\varepsilon,y} \geq 0 \quad \text{and} \quad \int_{\R^d} h_{\varepsilon,y}(x) \mu(dx) = 1, \quad \text{for all}  \ \varepsilon >0 \ \text{and} \ y \in \R^d.
\]
Let $|y| > \rho$ and $x \in B_{1}(y) \cap B_{|y|}(0)^c$, and consider $\widetilde y:= ((|y|-1/2)/|y|) y$. Observe that by the semigroup property and \eqref{eq:est_noniuc_proof_low} we have
\begin{align} \label{eq:Qt_estimate}
Q_t h_{\varepsilon,y}(x) = q_{(1+\varepsilon) t}(x,y) & \geq c_1 \frac{e^{(1+\varepsilon)\lambda_0 t}}{f_1^2(|y|)} \int_{B_{1/2}(\widetilde y)} f_1(|x-z|)f_1(|z-y|) e^{-(1+\varepsilon)\widetilde K t g(|z|)} dz \nonumber \\
& \geq c_1 \, |B_{1/2}(0)| \, f_1(1)f_1(2) \, \frac{e^{(1+\varepsilon)\lambda_0 t}}{f_1^2(|y|)} e^{-(1+\varepsilon)\widetilde K t g(|y|)} \\
&  \geq c_2 \, \frac{e^{(1+\varepsilon)\lambda_0 t}}{f_1^2(|y|)} e^{-(1+\varepsilon)\widetilde K t g(|y|)}, \nonumber
\end{align} 
with $c_2=c_1 |B_{1/2}(0)| f_1(1)f_1(2)$. 

Recall that the map $r \mapsto G_t(r): = f^2(r) \exp(\widetilde K t g(r))$ is continuous and, by \eqref{eq:a_b_event_decr}, it is eventually strictly decreasing on $(0,\infty)$.
In particular, there exists $r_0(t)>\rho$ such that $G_t$ is invertible on $[r_0(t),\infty)$.  Moreover, we may assume that $r_0(t)$ is chosen large enough so that $f(r)=f_1(r)$ and $g(r) \geq 1/\widetilde K$ for $r \geq r_0(t)$. We also define $\widetilde \kappa(t) = c_2 e^{-(1+2|\lambda_0|)t}$. 

Let
\[
\beta_t(u) := G_t^{-1} \left(\frac{\widetilde \kappa(t)}{u}\right), \quad u \geq u_0(t), \quad \text{with} \quad u_0(t):= \frac{\widetilde \kappa(t)}{G_t(r_0(t))}, 
\]
and further, assume that $y = y(u)$ is such that $|y| = \beta_t(u)$, $u \geq u_0(t)$. In particular,
	\begin{align}\label{eq:choice_of_y}
		\frac{\widetilde \kappa(t)}{u} = f_1^2(|y|) \exp(\widetilde K t g(|y|)).
	\end{align}
    
    We can now proceed to the conclusion.
    By taking $\varepsilon = \varepsilon(u) = 1/(\widetilde K g(|y|) \leq 1$, $u \geq u_0(t)$, we have
    \begin{equation}\label{ieq:exp}
        e^{(1+\varepsilon)\lambda_0 t} e^{-(1+\varepsilon)\widetilde K t g(|y|)} \geq e^{-(1+\varepsilon)|\lambda_0| t} e^{-t} e^{-\widetilde K t g(|y|)} \geq e^{-2|\lambda_0| t} e^{-t} e^{-\widetilde K t g(|y|)} .
    \end{equation}
    Thus by \eqref{eq:Qt_estimate}, \eqref{eq:choice_of_y} and \eqref{ieq:exp}, for $u \geq u_0(t)$, we get
	\begin{align*}
	\mu\big(\{x \in \R^d: Q_t h_{\varepsilon,y}(x) \geq u\}\big) & \geq \mu\left(\left\{x \in B_{1}(y) \cap B_{|y|}(0)^c:  c_2 \, \frac{e^{-2|\lambda_0| t}}{f_1^2(|y|)} e^{-t} e^{-\widetilde K t g(|y|)}\geq u\right\}\right) \\
                &   =  \mu\left(\left\{x \in B_{1}(y) \cap B_{|y|}(0)^c: \frac{\widetilde \kappa(t)}{f_1^2(|y|)} e^{-\widetilde K t g(|y|)}\geq u\right\}\right) \\
	              & =  \I_{\left\{ \,\frac{\widetilde \kappa(t)}{f_1^{2}(|y|)} \exp(-\widetilde K t g(|y|))\geq u\right\}} \mu\left(B_{1}(y) \cap B_{|y|}(0)^c\right).
	\end{align*}
Observe that, by \eqref{eq:choice_of_y}, the indicator function on the right hand side is equal to one as long as $u \geq u_0(t)$.
	Hence 
		\begin{align*}
\sup_{ h \geq 0 \atop \int h d\mu = 1}  \mu\big(\{x \in \R^d: Q_t h(x) \geq u\}\big) \geq \mu\left(B_{1}(y) \cap B_{|y|}(0)^c\right), \quad u \geq u_0(t).
	\end{align*}
It is easy to observe that $B_1(y) \cap B_{|y|}(0)^c \supseteq B_{1/2}$, where $B_{1/2}$ is ball of radius $1/2$. Consequently, by the lower bound in \eqref{eq:gs}, \eqref{eq:f_comp} and \eqref{A3}, we get
	\begin{align*}
		  	\sup_{ h \geq 0 \atop \int h d\mu = 1} \mu\big(\{x \in \R^d: Q_t h(x) \geq u\}\big) & \geq  \int_{ B_{1/2}} \varphi_0^2(z) dz \\
				 & \geq c_3 \int_{B_{1/2}} \frac{f_1^2(|z|)}{g^2(|z|)} dz
				   \geq c_4 \frac{f_1^2(|y|)}{g^2(|y|)}.		
	\end{align*}
	One more use of \eqref{eq:choice_of_y} reveals
	\[
	\frac{f_1^2(|y|)}{g^2(|y|)} = \frac{\widetilde \kappa(t)}{u} \frac{1}{g^2(\beta_t(u)) \exp(\widetilde K t g(\beta_t(u)))},
	\]
which leads us to a conclusion. 
\end{proof}

The following lemma provides a general and easy-to-verify sufficient condition under which the functions $u \mapsto g(\alpha_t(u)), g(\beta_t(u))$ are asymptotically equivalent. It will also be used in Section \ref{sec:lower_bound_orlicz} to analyze the behaviour of the function $u \mapsto g(\gamma_t(u))$ appearing in Theorem \ref{prop:main3}.

\begin{lemma}\label{lem:lemma} 
	Assume \eqref{A}. Suppose $F:(0,\infty) \to (0,\infty)$ is an eventually decreasing function for which there exists $r_0 > 0$ and $0<a<2$ such that
	\begin{equation*}
		f^2(r) \leq F(r) \leq f^a(r), \quad r > r_0.
	\end{equation*}
    Let $\widetilde \kappa >0$ be a constant and let $\tau(u) := F^{-1} \Big(\frac{\widetilde \kappa}{u}\Big)$, for $u$ large enough. 
    
	If $m:(0,\infty) \to (0,\infty)$ is a function which is eventually  increasing and $\omega \geq 0$ is a number such that for every  $c>0$ and $\lambda \geq 1$ we have
	\begin{align} \label{eq:cond_rate_v2_2}
		\limsup_{s \to 0^{+}} \frac{m\big(f^{-1}(c s^{\lambda})\big)}{m\big(f^{-1}(s)\big)} \leq \lambda^{ \omega},
	\end{align} 
	then, for any fixed $t > 0$,
	\begin{align*} 
		1 \leq \liminf_{u \to \infty} \frac{m(\tau(u))}{m(\alpha_t(u))} \leq \limsup_{u \to \infty} \frac{m(\tau(u))}{m(\alpha_t(u))} \leq \Big(\frac{2}{a}\Big)^{\omega}.
	\end{align*}
\end{lemma}

\begin{proof}
	First we observe that $(f^b)^{-1}(s) = f^{-1}(s^{1/b})$, $b>0$.
	Thus, there exists $s_0> 0$ such that
	\begin{equation}\label{ieq:inverses2}
		f^{-1} \big( s^{\frac{1}{2}}\big) \leq F^{-1} (s) \leq f^{-1} \big( s^{\frac{1}{a}}\big), \quad s \in (0, s_0).
	\end{equation}
	Fix $ t > 0$ and recall the definition \eqref{def:alpha} of the function $\alpha_t$, for some constant $\kappa(t)>0$. As $t$ is fixed, below we write $\kappa = \kappa(t)$ for shorthand. By the eventual monotonicity of $m$ and \eqref{ieq:inverses2}, for sufficiently large $u>0$, we have
	\[
	m \circ f^{-1}\left(\left(\tfrac{\widetilde \kappa}{u}\right)^\frac{1}{2}\right)
	\leq m \circ F^{-1}\left(\tfrac{\widetilde \kappa}{u}\right) 
	\leq m \circ f^{-1}\left(\left(\tfrac{\widetilde \kappa}{u}\right)^{\frac{1}{a}}\right) = m \circ f^{-1}\left(\left(\tfrac{\widetilde \kappa}{\kappa}\right)^{\frac{1}{a}}\left(\tfrac{\kappa}{u}\right)^{\frac{1}{a}}\right).
	\]
	We can write
	\[
	m(\alpha_t(u)) = m \circ f^{-1}\left(\left(\tfrac{ \kappa}{u}\right)^\frac{1}{2}\right) = m \circ f^{-1}\left(\left(\tfrac{ \kappa}{\widetilde\kappa}\right)^\frac{1}{2}\left(\tfrac{\widetilde \kappa}{u}\right)^{\frac{1}{2}}\right)
	\]
	and further
	\[
	\frac{m \circ f^{-1}\left(\left(\tfrac{\widetilde \kappa}{u}\right)^\frac{1}{2}\right)}{m \circ f^{-1}\left(\left(\tfrac{ \kappa}{\widetilde\kappa}\right)^{\frac{1}{2}}\left(\tfrac{\widetilde \kappa}{u}\right)^{\frac{1}{2}}\right)}
	\leq \frac{m(\tau(u))}{m(\alpha_t (u))} 
	\leq \frac{m\circ f^{-1}\left(\left(\tfrac{\widetilde \kappa}{\kappa}\right)^{\frac{1}{a}}\left(\tfrac{\kappa}{u}\right)^{\frac{1}{2}\frac{2}{a}}\right)}{m \circ f^{-1}\left(\left(\tfrac{\kappa}{u}\right)^{\frac{1}{2}}\right)}.
	\]
	Applying \eqref{eq:cond_rate_v2_2} with $c = ( \widetilde \kappa / \kappa)^{\frac{1}{a}}$
	and $\lambda = 2/a$, we get
	\[
	\limsup_{u \to \infty} \frac{m(\tau(u))}{m(\alpha_t (u))} \leq \ \Big(\frac{2}{a}\Big)^{\omega}.
	\]
	Similarly, by using \eqref{eq:cond_rate_v2_2} with $c = ( \kappa / \widetilde\kappa)^{\frac{1}{2}}$
	and $\lambda = 1$, we obtain
	\[
	\liminf_{u \to \infty} \frac{m(\tau(u))}{m(\alpha_t (u))} = \left(\limsup_{u \to \infty} \frac{m(\alpha_t (u))}{m(\tau(u))}\right)^{-1} \geq 1.
	\]
	This completes the proof. 
\end{proof}

We are now in a position to establish \eqref{eq:slow_rate_intro} and \eqref{eq:res_rate}.

\begin{lemma}\label{lemma:lem2}
		Assume \eqref{A}, \eqref{eq:decr_to_zero} and fix $ t > 0$. 
\begin{itemize}
\item[(a)] Then for every $\varepsilon \in (0,1)$ there exists $r_{\varepsilon} >0 $ such that 
	\begin{equation}\label{ieq:inverses_1}
		f^{2} (r) \leq G_t(r) \leq f^{2-\varepsilon} (r), \quad r > r_{\varepsilon},
	\end{equation}
    and, for every $\delta>0$, we have
    \[
    \lim_{u \to \infty} \frac{w_t\big(\alpha_t(u)\big)}{u^{\delta}}  = \lim_{u \to \infty} \frac{\widetilde w_t\big(\beta_t(u)\big)}{u^{\delta}} = 0.
    \]
\item[(b)] If, in addition, the potential profile $g$ satisfy condition \eqref{eq:cond_rate_v2_2} with some $\omega \geq 0$, then 
	\begin{align*}
		\lim_{u \to \infty} \frac{g(\beta_t(u))}{g(\alpha_t(u))} = 1.
	\end{align*}
    \end{itemize}
\end{lemma}

\begin{proof}
	Fix $ t > 0$. We first prove (a). Let $\varepsilon \in (0,1)$.
	From \eqref{eq:decr_to_zero} there is $r_{\varepsilon} >0$ such that
	\begin{equation*}
		\log(f(r)) < 0 \quad \text{and} \quad \frac{\widetilde K t g(r)}{-\log(f(r))} \leq \varepsilon, \quad r > r_{\varepsilon},
	\end{equation*} 
	which, after simple manipulations, becomes
	\begin{equation}\label{eq:exponent}
		e^{\widetilde K t g(r)} \leq \big(f (r)\big)^{-\varepsilon}, \quad r > r_{\varepsilon}.
	\end{equation} 
	Hence, from the definition of the function $G_t$ and \eqref{eq:exponent} we get
	\begin{equation*}
		f^2 (r) \leq G_t(r)= f^2(r) \exp(\widetilde K t g(r)) \leq f^{2 - \varepsilon} (r), \quad r > r_{\varepsilon},
	\end{equation*}
    which is the first assertion of Part (a). In order to show the second assertion, we fix $\delta>0$. By using the fact that $\alpha_t(u) \to \infty$ as $u \to \infty$ and the argument leading to \eqref{eq:exponent}, we obtain that for sufficiently large $u$,
	\begin{align*}
		w_t(\alpha_t(u)) & = g^2(\alpha_t(u)) \exp(Kt g(\alpha_t(u))) \\ 
		& \leq \exp((Kt+1) g(\alpha_t(u))) \\
		& \leq \big(f(\alpha_t(u))\big)^{-\delta}
		= \left(f\left( f^{-1}\left(\left(\frac{\kappa(t)}{u}\right)^{\frac{1}{2}}\right)\right)\right)^{-\delta}  = \left(\frac{u}{\kappa(t)}\right)^{\frac{\delta}{2}}.
	\end{align*}
    By the same argument and \eqref{ieq:inverses_1} (applied with $\varepsilon =1/2$) we also get 
    \begin{align*}
		\widetilde w_t(\beta_t(u))\leq \left(f\left( f^{-1}\left(\left(\frac{\widetilde \kappa(t)}{u}\right)^{\frac{1}{2-1/2}}\right)\right)\right)^{-\delta}  = \left(\frac{u}{\widetilde \kappa(t)}\right)^{\frac{2\delta}{3}}.
	\end{align*}
    Thus
	\begin{align*}
		0 \leq \frac{w_t\big(\alpha_t(u)\big)}{u^{\delta}} \leq \Big(\frac{1}{\kappa(t)u}\Big)^{\frac{\delta}{2}} \to 0 \quad \text{as} \quad u \to \infty,
	\end{align*}
    and 
    \begin{align*}
		0 \leq \frac{\widetilde w_t(\beta_t(u))}{u^{\delta}} \leq \Big(\frac{1}{\widetilde \kappa^2(t)u}\Big)^{\frac{\delta}{3}} \to 0 \quad \text{as} \quad u \to \infty.
	\end{align*}
   This completes the proof of Part (a). 

Part (b) follows directly from Lemma \ref{lem:lemma} applied to $F=G_t$ and $\widetilde \kappa = \widetilde \kappa(t)$ (due to the inequalities \eqref{ieq:inverses_1} established in Part(a)), as well as $m=g$. More precisely, we obtain 
	\begin{align*} 
		1 \leq \liminf_{u \to \infty} \frac{g(\beta_t(u))}{g(\alpha_t(u))} \leq \limsup_{u \to \infty} \frac{g(\beta_t(u))}{g(\alpha_t(u))} \leq \Big(\frac{2}{2-\varepsilon}\Big)^{\omega},
	\end{align*}
for every $\varepsilon \in (0,1)$.
\end{proof}

	\section{Proof of Theorem \ref{prop:main3}, Corollary \ref{cor:cor1} and Lemma \ref{lemma:lem3}}\label{sec:lower_bound_orlicz}

\begin{proof}[Proof of Theorem \ref{prop:main3}]
Assume \eqref{A} and \eqref{eq:decr_to_zero}.  Let $\eta, \sigma:[0,\infty) \to (0,\infty)$ be increasing and continuous functions such that $1/\eta \in L^1((1,\infty),dx)$, $\sigma(r) < r$, for $r \geq 1$, and \eqref{eq:eta_decr_to_zero}, \eqref{theta-bling} hold.  Define:
\[h(x) = \frac{g^2(|x|)}{\eta(|x|)|x|^{d-1}  f_1^2(|x|)}, \quad x \in \R^d.\] 
Due to \eqref{eq:gs}, it is direct to check that $h \in L^1(\mu)$.

Suppose that  $t>0$, $\rho>1$, and $c_1=c_1(t), \widetilde K>0$ are as in estimate \eqref{eq:est_noniuc_proof_low}.  Consider arbitrary $r$ such that $ \sigma(r)> \rho+1$. For any $x \in \R^d$ such that $\sigma(r) < |x| < r$, by \eqref{eq:est_noniuc_proof_low}, we have
\begin{align*} 
Q_t h(x) & = \int_{\R^d} q_t(x,y) h(y) \varphi_0^2(y)dy \nonumber\\
& \geq \frac{c_1 e^{\lambda_0 t}}{f_1(|x|)}\int_{\genfrac{}{}{0pt}{2}{|y-x|<1}{|y|<|x|}}\left(\frac{1}{f_1(|y|)} \int_{\rho-1 < |z| < |x|} f_1(|x-z|)f_1(|z-y|) e^{-\widetilde K t g(|z|)} dz \right)h(y) \varphi_0^2(y)dy  \nonumber \\
& \geq \frac{c_1 e^{\lambda_0 t}}{f_1(|x|)}\int_{\genfrac{}{}{0pt}{2}{|y-x|<1 }{|y|<|x|}}\left(\frac{1}{f_1(|y|)} \int_{\genfrac{}{}{0pt}{2}{ |z-x|<1}{ |z| < |x|}} f_1(|x-z|)f_1(|z-y|) e^{-\widetilde K t g(|z|)} dz \right)h(y) \varphi_0^2(y)dy  \nonumber \\
& \geq \frac{c_1 e^{\lambda_0 t}}{ f_1(\sigma(r))f_1(\sigma(r)-1)}   |B_{1/2}(0)| f_1(1)f_1(2) e^{-\widetilde K t g(r)} \int_{\genfrac{}{}{0pt}{2}{|y-x|<1 }{ |y|<|x|}} h(y) \varphi_0^2(y)dy.
\end{align*}
Here, we simply reduced the domain of integration in the second and third lines, and in the fourth line used both the monotonicity of the profiles $f_1$ and $g$, as well as the fact that there exists an Euclidean ball $B_{1/2}$ of radius $1/2$ such that $B_{1/2} \subset \left\{z \in \R^d: |z-x|<1, |z| < |x|\right\}$. 

We only need to estimate the last integral. By the definition of the function $h$ and \eqref{eq:gs}, 
\begin{align*}
\int_{\genfrac{}{}{0pt}{2}{|y-x|<1 }{ |y|<|x|}} h(y) \varphi_0^2(y)dy & \geq C_{10}^{-1} \int_{\genfrac{}{}{0pt}{2}{|y-x|<1}{ |y|<|x|}}\frac{g^2(|y|)}{\eta(|y|)|y|^{d-1}  f_1^2(|y|)} \frac{f_1^2(|y|)}{g^2(|y|)} dy \\ & = C_{10}^{-1} \int_{\genfrac{}{}{0pt}{2}{|y-x|<1}{ |y|<|x|}}\frac{1}{\eta(|y|)|y|^{d-1}} dy > C_{10}^{-1} \frac{1}{\eta(r)r^{d-1}} |B_{1/2}(0)|,
\end{align*}
by the monotonicity of the function $\eta$ and the reduction of the integration domain to a ball $B_{1/2}$, as before.
Since, by \eqref{eq:f_comp} and \eqref{theta-bling}, we have $f_1(\sigma(r))f_1(\sigma(r)-1) \leq C_7 f_1^2(\sigma(r)) \leq  c_2^2 C_7 f_1^2(r)$  (with $c_2$ coming from \eqref{theta-bling}),  it follows that for every $r > \sigma^{-1}(\rho+1)$ and $x \in \R^d$ such that $\sigma(r) < |x| < r$, 
\begin{align}\label{eq:crucial}
Q_t h(x) >  \frac{\widetilde \kappa(t)}{H_t(r)}, 
\end{align}
where $H_t(r):= f^2(r) r^{d-1}\eta(r)  \exp(\widetilde K t g(r))$, and $\widetilde \kappa(t):= c_1 |B_{1/2}(0)|^2  (c_2^2C_7C_{10})^{-1} f_1(1)f_1(2) e^{\lambda_0 t}$. 

We now analyze the invertibility of the function $r \mapsto H_t(r)$ for large values of $r$. We will adjust the argument leading to \eqref{eq:a_b_event_decr}. This function is continuous and satisfies
\begin{align*}
 H_t(r) & = \exp\left(2 \log f(r) + (d-1)\log r + \log \eta(r) + \widetilde K t g(r) \right) \\
        & = \exp\left(-|\log f(r)|\Big(2 - \frac{(d-1)\log r + \log \eta(r)}{|\log f(r)|} - \frac{ \widetilde K  t g(r)}{|\log f(r)|}\Big)\right).
\end{align*}
Since the map $r \mapsto |\log f(r)|$ is eventually strictly increasing by assumption, and 
\[
r \mapsto 2 - \frac{(d-1)\log r + \log \eta(r)}{|\log f(r)|} - \frac{ \widetilde K  t g(r)}{|\log f(r)|}
\]
eventually increases to the positive number
\[
2-\limsup_{r \to \infty} \frac{(d-1)\log r + \log \eta(r)}{|\log f(r)|} 
\]	
by \eqref{eq:decr_to_zero} and \eqref{eq:eta_decr_to_zero}, it follows that the map $r \mapsto H_t(r)$ is eventually strictly decreasing on $(0,\infty)$. Hence, there exists $r_0(t) > \sigma^{-1}(\rho+1)$ such that it is invertible on $[r_0(t),\infty)$ (we take $r_0(t)$ so large that $f_1(r)=f(r)$, for $r \geq r_0$). In particular, we can define:  
\[
  \gamma_t(u) := H_t^{-1} \left(\frac{\widetilde \kappa(t)}{u}\right), \quad u \geq u_0(t):= \frac{\widetilde \kappa(t)}{H_t(r_0(t))}.
\]
Further, we assume that $r = r_t(u) = \gamma_t(u)$, for $u \geq u_0(t)$, so that 
	\begin{align}\label{eq:choice_of_r}
		\frac{\widetilde \kappa(t)}{u} = f_1^2(r) r^{d-1}\eta(r)  \exp(\widetilde K t g(r)).
	\end{align}
Consequently, by \eqref{eq:crucial}, \eqref{eq:choice_of_r}, the lower bound in \eqref{eq:gs} and the monotonicity of $f_1$ and $g$, we get
	\begin{align*}
		  	\mu\big(\{x \in \R^d: Q_t h(x) > u\}\big) & \geq \mu\left(\left\{x \in \R^d: {\sigma(r)} < |x| < r, \frac{\widetilde \kappa(t)}{u} \geq H_t(r)\right\}\right) \\
				& = \mu\left(\left\{x \in \R^d: {\sigma(r)} < |x| < r\right\}\right) \\
				& =  \int_{{\sigma(r)} < |x| < r} \varphi_0^2(x) dx \\
				& \geq c_3 \int_{{\sigma(r)} < |x| < r} \frac{f_1^2(|x|)}{g^2(|x|)}dx  \\
				& \geq c_4 \frac{f_1^2(r)}{g^2(r)} r^{d-1}{(r - \sigma(r))},		
	\end{align*}
	as long as $u \geq u_0(t)$. By \eqref{eq:choice_of_r},
	\[
	\frac{f_1^2(r)}{g^2(r)} r^{d-1} = \frac{\widetilde \kappa(t)}{u} \frac{1}{g^2(\gamma_t(u)) \exp(\widetilde K t g(\gamma_t(u))) \eta(\gamma_t(u))},
	\]
which leads us to a conclusion. 
\end{proof}

\begin{proof}[Proof of Corollary \ref{cor:cor1}]
For any $h \in L^1(\mu)$, $h \geq 0$, and $\lambda>0$ we have
\begin{align*}
	\int_{\R^d} \Phi(h(x)/ \lambda) \mu(dx) = \int_{\R^d} \int_0^{h(x)/\lambda} \Phi^{\prime}(u) du \, \mu (dx)
	                 = \int_0^{\infty} \Phi^{\prime}(u) \mu\left(\left\{x \in \R^d: h(x) > \lambda u \right\}\right)du.
\end{align*}
Since $\mu$ is a probability measure and $\Phi^{\prime}$ is integrable on any interval $[0,a]$, $a>0$, the both assertions (a) and (b) follow directly from the respective estimates of $\mu\left(\left\{x \in \R^d: Q_t h(x) > \lambda u \right\}\right)$ in Theorems \ref{thm:main} and \ref{prop:main3}. 
\end{proof}

We conclude this section by analyzing the behavior of the rate $v_t(\gamma_t(u))$ appearing in Theorem \ref{prop:main3} and Corollary \ref{cor:cor1}, and its exponent. 

\begin{lemma}\label{lemma:lem3}
		Assume that \eqref{A}, \eqref{eq:decr_to_zero} hold and that the potential profile $g$ satisfies condition \eqref{eq:cond_rate_v2_2} with some $\omega \geq 0$. Let $ t >0$ and let $\eta, \sigma:[0,\infty) \to (0,\infty)$ be continuous and increasing functions such that $1/\eta \in L^1((1,\infty),dx)$, $\sigma(r) < r$, for $r \geq 1$, and \eqref{eq:eta_decr_to_zero}, \eqref{theta-bling} hold. Denote
        \[
           b := \lim\limits_{r \to \infty} \frac{(d-1)\log r + \log \eta(r)}{|\log f(r)|} \in [0,2), \quad \text{see \eqref{eq:eta_decr_to_zero}.}
        \]
        Then the following assertions hold. 
        \begin{itemize}
            \item[(a)]
       For every $\varepsilon \in (0,2-b)$ there exists $r_{\varepsilon} >0 $ such that 
	\begin{equation}\label{ieq:inverses_2}
		f^2 (r) \leq H_t (r) \leq f^{2- b -\varepsilon} (r), \quad r > r_{\varepsilon}, 
	\end{equation}
    \begin{align}\label{eq:gfraca}  
	 	1 \leq \liminf_{u \to \infty} \frac{g(\gamma_t(u))}{g(\alpha_t(u))} \leq \limsup_{u \to \infty} \frac{g(\gamma_t(u))}{g(\alpha_t(u))} \leq  \Big(\frac{2}{2-b}\Big)^{\omega}.
	 \end{align}
    and, for every $\delta>0$, we have
	\begin{align}\label{eq:slow_rate_2}
		\lim_{u \to \infty} \frac{\widetilde w_t\big(\gamma_t(u)\big)}{u^{\delta}}= 0.
	\end{align}
 
 	\item[(b)]  If, in addition, 
 	\begin{equation*}
 		\lim\limits_{r \to \infty}\frac{\log \eta(r) - \log(r-\sigma(r))}{g(r)} =0, 
 	\end{equation*}
    then
    \begin{align}
	 	\lim_{u \to \infty} \frac{\log \eta(\gamma_t(u)) - \log(\gamma_t(u)-\sigma(\gamma_t(u)))}{g(\alpha_t(u))} = 0
	 \end{align}
 	and, for any $\delta> 0$,
 	\begin{equation*}
 		\lim_{u \to \infty} \frac{v_t\big(\gamma_t(u)\big)}{u^{\delta}}= 0.
 	\end{equation*}
    \end{itemize}
\end{lemma}

\begin{proof}
	Fix $ t >0$. We first show (a). Let $\varepsilon \in (0,2-b)$. 
    By \eqref{eq:eta_decr_to_zero} there exists $r_{\varepsilon}>0$ such that 
	\begin{equation*}
		r^{d-1} \eta(r) \leq f(r)^{-(b+\varepsilon/2)}, \quad r > r_{\varepsilon}.
	\end{equation*}
	Using this inequality and the argument leading to \eqref{eq:exponent}, we get
	\begin{equation*}
		f^2 (r) \leq H_t(r)= f^2(r) r^{d-1}\eta(r) \exp(\widetilde K t g(r)) \leq f^{2 - b -  \varepsilon} (r), \quad r > r_{\varepsilon},
	\end{equation*}
    for sufficiently large $r_{\varepsilon} >0$. This is the first assertion of Part (a).
    
      In order to see \eqref{eq:gfraca}, we just apply Lemma \ref{lem:lemma} with $F=H_t$, $\widetilde \kappa = \widetilde \kappa(t)$ (it is allowed due to the inequalities \eqref{ieq:inverses_2} established above) and $m=g$ (here $\tau(u) = \gamma_t(u))$.
    
      The last assertion \eqref{eq:slow_rate_2} follows directly from \eqref{eq:gfraca} and $\lim_{u \to \infty} (w_t\big(\alpha_t(u)\big)/u^{\delta})  = 0$, which was established in Lemma \ref{lemma:lem2} (a). 

 The first assertion of (b) is a straightforward consequence of the assumption and \eqref{eq:gfraca}. The second assertion then follows from \eqref{eq:slow_rate_2} and the following estimate
    \[
0 \leq v_t(r) = g^2(r) \exp\left(\widetilde K t g(r)\left(1+\frac{\log \eta(r) - \log(r - \sigma(r) )}{\widetilde K t g(r)}\right)\right) 
  \leq \widetilde w_{2t}(r) \leq \widetilde w_{t}^2(r) ,
    \]
 which is valid for sufficiently large $r>0$.  
    
\end{proof}

\section{Applications and examples} \label{sec:ex}

We now illustrate our results from Theorems \ref{thm:main} and \ref{prop:main2} through several examples of ground state transformed semigroups $\{Q_t:t \geq 0\}$ associated with non-local Schr\"odinger operators of the form $H=-L+V$, where the kinetic term $L$ is taken to be either the  \emph{fractional Laplacian} $L=-(-\Delta)^a$, or the \emph{fractional relativistic Laplacian} $L=-(-\Delta+m^{1/a})^a+m$, with $a \in (0,1)$ and $m>0$. 
We also apply Corollary \ref{cor:cor1} to characterize the range $Q_t \big(L^1(\mu)\big)$ in terms of Orlicz spaces. For each operator $L$, we consider two distinct families of confining potentials $V$, which lead to fundamentally different decay rates and associated Young functions. 

In this connection, we refer to the recent work of Roberto and Zegarli\'nski \cite{RZ}, where Orlicz hypercontractivity is investigated for a broad class of diffusion semigroups (see also the earlier contribution \cite{BCR}). We further note that Young functions and Orlicz spaces have been previously studied in the setting of certain non-local operators, Dirichlet forms, and jump processes; see, for example, Bogdan, Kutek, and Pietruska-Pa{\l}uba \cite{BKPP}, Schilling and Wang \cite{bib:SchW}, Wang \cite{Wang}, and Wang and Wang \cite{WangWang}. These works focus on different objectives, but they also offer additional motivation for the applicability of our results.

\subsection{Intrinsic semigroups associated with fractional Schr\"odinger operators}

Let 
\[
L=-(-\Delta)^a, \quad a \in (0,1).
\]
We know that $\nu(dx) = \nu(x) dx$, with $\nu(x) = c_{d,a} |x|^{-d-2a}$, in the integral representation \eqref{eq:symbol}. Clearly, the profile of the L\'evy density is given by $f(r) =  r^{-d-2a}$ which satisfies \eqref{A1} \cite[Lemma 3.2(a)]{KSch}. Moreover, it is known that 
\[
p_t(x) \asymp t^{-d/(2a)} \wedge t |x|^{-d-2a}, \quad t>0, \ x \neq 0, 
\]
which in particular implies \eqref{A2}. 

Clearly, 
\[f^{-1}(s) = s^{-\frac{1}{d+2a}}\] 
and
\begin{equation*}
		\alpha_t(u) = (f^2)^{-1} \Big(\frac{\kappa(t)}{u}\Big) = f^{-1} \Big(\Big(\frac{\kappa(t)}{u}\Big)^{\frac{1}{2}}\Big) = \Big(\frac{\kappa(t)}{u} \Big)^{- \frac{1}{2(d+2a)}}
	\end{equation*}
for a suitable constant $\kappa(t)>0$. 

In Examples \eqref{ex:ex1} and \eqref{ex:ex2} we discuss two types of confining potentials which evidently satisfy \eqref{A3}.
	
\begin{example}[\textit{Power-logarithmic potentials}] \label{ex:ex1}
Let
\[
V(x) = 0 \vee \log^{ \theta}|x|, \quad \theta > 0.
\]
We choose the potential profile $g$ such that $g(r) = \log^{\theta} r$ for $r \geq e$.

\medskip

\noindent
\textbf{Case 1:}\ $\theta \geq 1$ $\Leftrightarrow$ $\exists C$ $g(r) \geq C|\log f(r)|$, $r \geq e$ $\Leftrightarrow$ $\{Q_t:t \geq 0\}$ is asymptotically ultracontractive/hypercontractive, see e.g.\ \cite[Corollary 3.3]{bib:KKL2018} or \cite{bib:ChW2015}; 

\medskip

\noindent
\textbf{Case 2:}\ $\theta < 1$ $\Leftrightarrow$ condition \eqref{eq:decr_to_zero} is satisfied. In this case, the following assertions hold:

\begin{itemize}
\item The estimates in Theorems \ref{thm:main} and \ref{prop:main2} hold with 
\[
w_t(\alpha_t(u)) = g^2(\alpha_t(u)) \cdot \exp\big(K t g(\alpha_t(u))\big) 
\]
and 
\[
\widetilde w_t(\beta_t(u)) = g^2(\beta_t(u)) \cdot \exp\left(\widetilde K t g(\beta_t(u))\right), 
\]
where
\begin{align} \label{eq:border_frac}
g(\alpha_t(u)) \approx g(\beta_t(u)) \approx \left(\frac{1}{2(d+2a)} \log u\right)^{\theta}, \quad \text{as} \ u \to \infty, \ \text{for every fixed} \ t > 0. 
\end{align}
According to the discussion at the end of Section \ref{sec:sharp}, 
a (far from optimal) estimate for the constants $K,\widetilde K$ is $K^{-1} = \widetilde K = 4\log^{2\theta}(1+e)$.

\item Let $c>0$ and $\Phi(u)=|u| \exp\left(c \log^{\theta} (e+|u|)\right)$ be a Young function. It follows from Corollary \ref{cor:cor1} that:
\begin{align} \label{eq:fract_orlicz_1}
 \text{if $c < Kt/(2d+4a)^{\theta}$, then $Q_t$  maps  $L^1(\mu)$  continuously into  $\mL^{\Phi}(\mu)$;}
\end{align}
\begin{align} \label{eq:fract_orlicz_2}
\text{if $c >  \widetilde Kt/(d+4a)^{\theta}$, then $Q_t \big(L^1(\mu)\big) \not \subset \mL^{\Phi}(\mu)$.} 
\end{align}
\end{itemize}

\medskip
\noindent 
\underline{Proof of \eqref{eq:border_frac}:}
Indeed, we have
	\begin{equation*}
		g(f^{-1}(s)) = \Bigg(-\frac{1}{d+2a} \log s\Bigg)^{\theta},
	\end{equation*}
for $s$ small enough, and
	\begin{equation} \label{eq:fract_alpha}
		g(\alpha_t (u)) = \log^{\theta}\!\alpha_t(u) = \Bigg(\frac{1}{2(d+2a)} \Big( \log u - \log \kappa(t) \Big)\Bigg)^{\theta} \approx \Bigg(\frac{1}{2(d+2a)} \log u\Bigg)^{\theta}.
	\end{equation}
Moreover, for every $c>0$ and $\lambda \geq 1$, 
	\begin{equation*}
		\frac{g(f^{-1}(c s^{\lambda}))}{g(f^{-1}(s))} = \Bigg(\frac{ \log c + \lambda \log s }{\log s}\Bigg)^{\theta} \xrightarrow{ s \to 0^+} \lambda^{\theta},
	\end{equation*}
which means that \eqref{eq:cond_rate_v2_2} with $m=g$ holds true. Consequently, by Lemma \ref{lemma:lem2}(b),
	\begin{equation*}
		g(\beta_t (u)) = \log^{\theta}\!\beta_t(u) \approx \Bigg(\frac{1}{2(d+2a)} \log u \Bigg)^{\theta}.
	\end{equation*}
This completes the proof of \eqref{eq:border_frac}.

\medskip
\noindent 
\underline{Proof of \eqref{eq:fract_orlicz_1}:}   
We apply Corollary \ref{cor:cor1}(a) in combination with \eqref{eq:fract_alpha} and the asymptotic estimate $\Phi^{\prime}(u) \approx  \exp\left(c \log^{\theta} |u|\right)$.  

\medskip
\noindent 
\underline{Proof of \eqref{eq:fract_orlicz_2}:} First, we choose the function $\eta$ in such a way that $\eta(r) = r \log^2 r$ for sufficiently large $r$, and define $\gamma_t(u)$ and $v_t(\gamma_t(u))$ as in Theorem \ref{prop:main3}. As we already know, the profile $g$ satisfies \eqref{eq:cond_rate_v2_2} and, consequently, Lemma \ref{lemma:lem3}(a) implies that 
\begin{align} \label{eq:aux_ex1}
	 	1 \leq \liminf_{u \to \infty} \frac{\log^{\theta}\gamma_t(u)}{\log^{\theta}\alpha_t(u)} \leq \limsup_{u \to \infty} \frac{\log^{\theta}\gamma_t(u)}{\log^{\theta}\alpha_t(u)} \leq  \Big(\frac{2d+4a}{d+4a}  \Big)^{\theta}.
\end{align}
Observe that $f$ satisfies \eqref{theta-bling} with $\sigma(r) = \frac{r}{2}$. In particular,
\[
\frac{\eta(r)}{r-\sigma(r)} = 2\log^2 r, \quad \text{for $r$ big enough,}
\]
which means that
\[
\frac{\eta(\gamma_t(u))}{\gamma_t(u)-\sigma(\gamma_t(u))} =2\log^2\gamma_t(u), \quad \text{for large $u$.}
\]
This shows that
\[
v_t(\gamma_t(u)) =  \exp\big(\widetilde K t \log^{\theta}\gamma_t(u)+(2\theta+2)\log \log\gamma_t(u)+ \log 2\big),
\]
for sufficiently large $u$. By combining this with \eqref{eq:aux_ex1} and \eqref{eq:border_frac}, we get
\[
v_t(\gamma_t(u)) \leq  \exp\big((1+o(1))\widetilde K t (d+4a)^{-\theta} \log^{\theta} u \big), \quad \text{as} \quad u \to \infty.
\]
Since $\Phi^{\prime}(u) \approx  \exp\left(c \log^{\theta} |u|\right)$, the assertion \eqref{eq:fract_orlicz_2} follows from Corollary \ref{cor:cor1}(b). 
\end{example} 

\begin{example}[\textit{Power-iterated-logarithmic potential}] \label{ex:ex2}
Let
\[
V(x) = 0 \vee \log^{\theta} \log|x|, \quad \theta > 0.
\]
We choose $g$ such that $g(r) = \log^{\theta} \log r$ for $r \geq e^e$. It is clear that $g$ satisfies \eqref{eq:decr_to_zero} for every $\theta>0$. The constants $K,\widetilde K$ can be chosen such that $K^{-1} = \widetilde K = 4\log^{2\theta} \log(1+e^e)$.
\begin{itemize}
\item 
The rates $w_t(\alpha_t(u))$ and $\widetilde w_t(\beta_t(u))$ in Theorems \ref{thm:main} and \ref{prop:main2} are governed by 
	\begin{align} \label{eq:loglog_ex}
		g(\alpha_t (u)) \approx g(\beta_t (u)) \approx \log^{\theta}\log u, \quad \text{as} \ u \to \infty, \ \text{for every fixed} \ t > 0.
	\end{align}
\item Let $c>0$ and let $\Phi(u)=|u| \exp\left(c \log^{\theta} \log (e^e+|u|)\right)$ be a Young function. An application of Corollary \ref{cor:cor1} yields the following two observations:

\smallskip
\noindent
\textbf{Case 1:}\ $\theta>1$

    \begin{align} \label{eq:fract_orlicz_32}
 \text{if $c < Kt$, then $Q_t$  maps  $L^1(\mu)$  continuously into  $\mL^{\Phi}(\mu)$;}
\end{align}
\begin{align} \label{eq:fract_orlicz_42}
\text{if $c > \widetilde Kt$, then $Q_t \big(L^1(\mu)\big) \not \subset \mL^{\Phi}(\mu)$.}
\end{align}

\smallskip
\noindent
\textbf{Case 2:}\ $\theta =1$

    \begin{align} \label{eq:fract_orlicz_3}
 \text{if $c < Kt -1$, then $Q_t$  maps  $L^1(\mu)$  continuously into  $\mL^{\Phi}(\mu)$;}
\end{align}
\begin{align} \label{eq:fract_orlicz_4}
\text{if $c > \widetilde Kt+1$, then $Q_t \big(L^1(\mu)\big) \not \subset \mL^{\Phi}(\mu)$.}
\end{align}

\smallskip
\noindent
For $\theta < 1$, the upper bound in Theorem~\ref{thm:main} is not integrable, and hence yields no information on the range $Q_t\big(L^1(\mu)\big)$.

\end{itemize}

\medskip
\noindent 
\underline{Proof of \eqref{eq:loglog_ex}:}
We have
    \begin{equation*}
		g(\alpha_t (u)) = \left(\log\Big(\log u  - \log \kappa(t) \Big) + \log(\tfrac{1}{2d+4a})\right)^{\theta} \approx \log^{\theta}\log u.
	\end{equation*}
Moreover, the condition \eqref{eq:cond_rate_v2_2} is satisfied with $m=g$:\ for every $\theta >0$ and $\lambda \geq 1$ we have
	\begin{equation*}
		\frac{g(f^{-1}(c s^{\lambda}))}{g(f^{-1}(s))} = \left(\frac{ \log\big(-\log s - \tfrac{1}{\lambda}\log c\big)   + \log(\tfrac{\lambda}{d+2a}) }{\log(-\log s) + \log(\tfrac{1}{d+2a}) }\right)^{\theta} \xrightarrow{ s \to 0^+} 1 .
	\end{equation*}
Hence, by Lemma \ref{lemma:lem2}(b), we have $g(\beta_t (u)) \approx \log^{\theta}\log u$, which completes the proof of \eqref{eq:loglog_ex}.

\medskip
\noindent 
\underline{Proof of \eqref{eq:fract_orlicz_32} and \eqref{eq:fract_orlicz_3}:} Similarly as in the previous example, this follows from Corollary \ref{cor:cor1}(a), \eqref{eq:loglog_ex} and the asymptotic estimate $\Phi^{\prime}(u) \approx  \exp\left(c \log^{\theta} \log |u|\right)$. 

\medskip
\noindent 
\underline{Proof of \eqref{eq:fract_orlicz_42} and \eqref{eq:fract_orlicz_4}:}
Here, it is convenient to choose the function $\eta$ in such a way that $\eta(r) = \eta(r) = r \log r \cdot \log^2 \log r$, for sufficiently large $r$. Recall that $\gamma_t(u)$ and $v_t(\gamma_t(u))$ are defined in Theorem \ref{prop:main3}. As pointed out above (see the proof of \eqref{eq:border_frac}), the condition \eqref{eq:cond_rate_v2_2} is satisfied with $m(s)= \log s$. Therefore, by applying Lemma \ref{lem:lemma} with $F=H_t$, we get 
\begin{align*}
	 	1 \leq \liminf_{u \to \infty} \frac{\log \gamma_t(u)}{\log\alpha_t(u)} \leq \limsup_{u \to \infty} \frac{\log\gamma_t(u)}{\log\alpha_t(u)} \leq  2.
\end{align*}
Consequently, by \eqref{eq:loglog_ex}, 
\begin{align}\label{ex2:aux}
    g(\gamma_t(u)) = \log^{\theta} \log \gamma_t(u) \approx \log^{\theta} \log u, \quad \text{as} \ u \to \infty.
\end{align}
Since \eqref{theta-bling} holds with $\sigma(r) = \frac{r}{2}$, we have
\[
\frac{\eta(r)}{r-\sigma(r)} = 2\log r \log^2 \log r, \quad \text{for large $r$,}
\]
that is
\[
\frac{\eta(\gamma_t(u))}{\gamma_t(u)-\sigma(\gamma_t(u))} =2\log \gamma_t(u) \log^2 \log \gamma_t(u), \quad \text{for large $u$.}
\]
Hence
\[
v_t(\gamma_t(u)) =  \exp\big(\widetilde K t \log^{\theta} \log \gamma_t(u)+(2\theta+2)\log \log \log\gamma_t(u) + \log \log \gamma_t(u) +\log 2\big),
\]
for sufficiently large $u$ and, by \eqref{ex2:aux}, 
\[
v_t(\gamma_t(u)) \leq  \exp\big((1+o(1))\widetilde K t \log^{\theta} \log u\big), \quad \text{as} \quad u \to \infty,
\]
whenever $\theta>1$, and
\[
v_t(\gamma_t(u)) \leq  \exp\big((1+o(1))(\widetilde K t+1) \log \log u\big), \quad \text{as} \quad u \to \infty,
\]
whenever $\theta = 1$. 

Now, as $\Phi^{\prime}(u) \approx  \exp\left(c \log^{\theta} \log |u|\right)$, the assertions  \eqref{eq:fract_orlicz_42} and \eqref{eq:fract_orlicz_4} follow from Corollary \ref{cor:cor1}(b). 
\end{example}

\subsection{Intrinsic semigroups associated with fractional relativistic Schr\"odinger operators}
Let 
\[
L=-(-\Delta+m^{1/a})^a+m, \quad a \in (0,1), \ \ m>0.
\]
One has $\nu(dx) = \nu(x) dx$ in \eqref{eq:symbol}, where
\begin{align*}
	\nu(x) &= \frac{a}{(4\pi)^{d/2}\Gamma\left(1-a\right)} \int_0^{\infty} \exp\left(-\frac{|x|^2}{4u} - m^{\frac 1a} u\right) u^{-1-\frac{d+2a}{2}} \,du \\
    &= \frac{a 2^{1+a-\frac{d}{2}} m^{\frac{d+2a}{4a}}}{\pi^{\frac{d}{2}}\Gamma\left(1-a\right)} \frac{K_{\frac{d+2a}{2}}\left(m^{\frac{1}{2a}}|x|\right)}{|x|^{\frac{d+2a}{2}}},
\end{align*} 
where
\begin{gather*}
	K_{\mu}(r) 
	= \frac{1}{2} \left(\frac{r}{2}\right)^{\mu} \int_0^{\infty} u^{-\mu-1} \exp\left(-u-\frac{r^2}{4u}\right) du, 
	\quad \mu>0,\; r>0,
\end{gather*} 
is the modified Bessel function of the second kind, see e.g.\ \cite[10.32.10]{NIST}. Using the asymptotics (see \cite[10.25.3 and 10.30.2]{NIST})
\begin{align*}
	\lim_{r \to \infty} K_{\mu}(r) \sqrt{r} e^r  = \sqrt{\pi/2}, 
	\qquad \lim_{r \to 0} K_{\mu}(r) r^{\mu}  = 2^{\mu-1} \Gamma(\mu),
\end{align*}
we can show that the profile $f$ of the L\'evy density is given by
\begin{align*}
	f(r) = \left(\I_{[0,1]}(r) r^{-d-2a} + \I_{(1,\infty)}(r) r^{-(d+2a+1)/2} \right)\exp\left(-m^{\frac{1}{2a}}r\right).
\end{align*}
A concise verification of \eqref{A1} is given in \cite[Section~5.5]{KSch}, while \eqref{A2} follows directly from \cite[Theorem~2]{KSz}.

Observe that the function $p(s) = -(1/m^{\frac{1}{2a}}) \log s$ serves as a left asymptotic inverse of $f$. Indeed, 
	\begin{align*}
		\frac{p(f(r))}{r}  = \frac{-1}{m^{\frac{1}{2a}}r} \log(e^{-m^{\frac{1}{2a}}r} r^{-(d+2a+1)/2})
		  = \frac{m^{\frac{1}{2a}} r + \frac{d+2a+1}{2} \log r}{m^{\frac{1}{2a}}r}  \xrightarrow{r \to \infty} 1.
	\end{align*}
	In particular,
	\begin{equation}\label{eq:asymp_of_p}
		\lim\limits_{s \to 0^+}\frac{p(s)}{f^{-1}(s)} = \lim\limits_{r \to \infty}\frac{p(f(r))}{f^{-1}(f(r))} = \lim\limits_{r \to \infty}\frac{p(f(r))}{r} = 1.
	\end{equation} 
	This implies that
	\[
		\alpha_t(u) =  f^{-1} \Big(\Big(\frac{\kappa(t)}{u}\Big)^{\frac{1}{2}}\Big)  \approx p \Big(\Big(\frac{\kappa(t)}{u}\Big)^{\frac{1}{2}}\Big) \approx \frac{1}{2m^{1/(2a)}} \log u,
	\]
	for a suitable constant $\kappa(t)>0$, for every fixed $t>0$. Similarly to the fractional Laplacian case, in Examples \eqref{ex:ex3} and \eqref{ex:ex4} we analyze two different classes of confining potentials that satisfy \eqref{A3}.

\begin{example}[\textit{Power-type potential}] \label{ex:ex3}
Let
\[
V(x) = |x|^{\theta}, \quad \theta > 0.
\]
We take the potential profile $g$ such that $g(r) = r^{\theta}$ for $r\geq 1$. We have two cases:

\medskip

\noindent
\textbf{Case 1:}\ $\theta \geq 1$ $\Leftrightarrow$ $\exists C$ $g(r) \geq C|\log f(r)|$, $r \geq 1$ $\Leftrightarrow$ $\{Q_t:t \geq 0\}$ is asymptotically ultracontractive/hypercontractive, see e.g.\ \cite[Corollary 3.3]{bib:KKL2018} or \cite{bib:ChW2015}; 

\medskip

\noindent
\textbf{Case 2:}\ $\theta < 1$ $\Leftrightarrow$ condition \eqref{eq:decr_to_zero} is satisfied. In this case, the following assertions hold:

\begin{itemize}
\item The estimates in Theorems \ref{thm:main} and \ref{prop:main2} hold with 
\[
w_t(\alpha_t(u)) = g^2(\alpha_t(u)) \cdot \exp\big(K t g(\alpha_t(u))\big) 
\]
and 
\[
\widetilde w_t(\beta_t(u)) = g^2(\beta_t(u)) \cdot \exp\left(\widetilde K t g(\beta_t(u))\right), 
\]
where
\begin{align} \label{eq:border_relativ}
g(\alpha_t(u)) \approx g(\beta_t(u)) \approx \left(\frac{1}{2m^{1/(2a)}} \log u\right)^{\theta}, \quad \text{as} \ u \to \infty, \ \text{for every fixed} \ t > 0. 
\end{align}
According to the discussion at the end of Section \ref{sec:sharp}, the constants $K,\widetilde K$ may be chosen such that $K^{-1} = \widetilde K = 4^{1+\theta}$.

\item Let $c>0$ and $\Phi(u)=|u| \exp\left(c \log^{\theta} (e+|u|)\right)$ be a Young function. By Corollary \ref{cor:cor1}:
\begin{align} \label{eq:fract_orlicz_5}
 \text{if $c < K t /(2m^{\frac{1}{2a}})^{\theta}$, then $Q_t$  maps  $L^1(\mu)$  continuously into  $\mL^{\Phi}(\mu)$;}
\end{align}
\begin{align} \label{eq:fract_orlicz_6}
\text{if $c >  \widetilde K t /(2m^{\frac{1}{2a}})^{\theta}$, then $Q_t \big(L^1(\mu)\big) \not \subset \mL^{\Phi}(\mu)$.} 
\end{align}
\end{itemize}

\medskip
\noindent 
\underline{Proof of \eqref{eq:border_relativ}:}
Clearly, 
	\begin{equation*}
	g(\alpha_t(u)) = \alpha_t^{\theta}(u) \approx \left(\frac{1}{2m^{1/(2a)}} \log u \right)^{\theta}.
	\end{equation*}
	Moreover, for every $c>0$ and  $\lambda \geq 1$, we get
	\begin{equation*}
		\lim\limits_{s \to 0^+}\ \frac{g(p(c s^{\lambda}))}{g(p(s))} = \lim\limits_{s \to 0^+}\ \Bigg(\frac{ \log c + \lambda \log s }{\log s}\Bigg)^{\theta} = \lambda^{\theta}.
	\end{equation*}
	Combining this with \eqref{eq:asymp_of_p}, we get
		\begin{equation*}
		\lim\limits_{s \to 0^+}\frac{g(f^{-1}(c s^{\lambda}))}{g(f^{-1}(s))} 
		= \lim\limits_{s \to 0^+} \left(\frac{f^{-1}(c s^{\lambda})}{p(c s^{\lambda})}\right)^{\theta}
		\left(\frac{p(s)}{f^{-1}(s)}\right)^{\theta} \frac{g(p(c s^{\lambda}))}{g(p(s))} =  \lambda^{\theta},
	\end{equation*}
	which shows that $g$ satisfies \eqref{eq:cond_rate_v2_2}. Hence, by Lemma \ref{lemma:lem2}(b), 
	\begin{equation*}
		g(\beta_t(u)) = \beta_t^{\theta} (u) \approx  \left(\frac{1}{2m^{1/(2a)}} \log u \right)^{\theta}.
	\end{equation*}

\medskip
\noindent 
\underline{Proof of \eqref{eq:fract_orlicz_5}:}   
As before, we apply Corollary \ref{cor:cor1}(a), \eqref{eq:border_relativ} and the asymptotic estimate $\Phi^{\prime}(u) \approx  \exp\left(c \log^{\theta} |u|\right)$.  

\medskip
\noindent 
\underline{Proof of \eqref{eq:fract_orlicz_6}:} 
We choose the function $\eta$ in such a way that $\eta(r) = r^2$, for sufficiently large $r$, and define $\gamma_t(u)$ and $v_t(\gamma_t(u))$ as in Theorem \ref{prop:main3}. The profile $g$ satisfies \eqref{eq:cond_rate_v2_2} and, consequently, Lemma \ref{lemma:lem3}(a) (where $b=0$) implies that 
\begin{align} \label{eq:aux_ex1b}
	 	g(\gamma_t(u)) \approx g(\alpha_t(u)) 
                    \approx  \left(\frac{1}{2m^{1/(2a)}} \log u \right)^{\theta}.
\end{align}
Moreover, we take $\sigma(r) = r-1$ in \eqref{theta-bling}, getting
\[
\frac{\eta(r)}{r-\sigma(r)} =r^2, \quad \text{for large $r$,}
\]
which implies that
\[
\frac{\eta(\gamma_t(u))}{\gamma_t(u)-\sigma(\gamma_t(u))} =\gamma^2_t(u), \quad \text{for large $u$.}
\]
Consequently, 
\[
v_t(\gamma_t(u)) =  \exp\big(\widetilde K t \gamma^{\theta}_t(u)+(2\theta+2)\log \gamma_t(u)\big),
\]
for sufficiently large $u$. Finally, by \eqref{eq:aux_ex1b}, we get
\[
v_t(\gamma_t(u)) \leq  \exp\big((1+o(1))\widetilde K t (2m^{\frac{1}{2a}})^{-\theta} \log^{\theta} u \big), \quad \text{as} \quad u \to \infty.
\]
Since $\Phi^{\prime}(u) \approx  \exp\left(c \log^{\theta} |u|\right)$, the assertion \eqref{eq:fract_orlicz_6} follows from Corollary \ref{cor:cor1}(b). 
\end{example}

\begin{example}[\textit{Power-logarithmic potential}] \label{ex:ex4}
Let
\[
V(x) = 0 \vee \log^{\theta}|x|, \quad \theta > 0.
\]
Let $g(r) = \log^{\theta} r$ for $r \geq e$. Then $g$ satisfies \eqref{eq:decr_to_zero} for every $\theta>0$. The constants $K,\widetilde K$ can be chosen such that $K^{-1} = \widetilde K = 4\log^{2\theta}(1+e)$. 

Let us summarize our findings:
\begin{itemize}
\item 
The rates $w_t(\alpha_t(u))$ and $\widetilde w_t(\beta_t(u))$ in Theorems \ref{thm:main} and \ref{prop:main2} are governed by 
	\begin{align} \label{eq:log_ex_last}
		g(\alpha_t (u)) \approx g(\beta_t (u)) \approx \log^{\theta}\log u, \quad \text{as} \ u \to \infty, \ \text{for every fixed} \ t > 0.
	\end{align}
\item Let $c>0$ and let $\Phi(u)=|u| \exp\left(c \log^{\theta} \log (e^e+|u|)\right)$ be a Young function. By Corollary \ref{cor:cor1} we get the following two cases:

\smallskip
\noindent
\textbf{Case 1:}\ $\theta>1$

    \begin{align} \label{eq:fract_orlicz_7}
 \text{if $c < Kt$, then $Q_t$  maps  $L^1(\mu)$  continuously into  $\mL^{\Phi}(\mu)$;}
\end{align}
\begin{align} \label{eq:fract_orlicz_8}
\text{if $c > \widetilde Kt$, then $Q_t \big(L^1(\mu)\big) \not \subset \mL^{\Phi}(\mu)$.}
\end{align}

\smallskip
\noindent
\textbf{Case 2:}\ $\theta =1$

    \begin{align} \label{eq:fract_orlicz_72}
 \text{if $c < Kt -1$, then $Q_t$  maps  $L^1(\mu)$  continuously into  $\mL^{\Phi}(\mu)$;}
\end{align}
\begin{align} \label{eq:fract_orlicz_82}
\text{if $c > \widetilde Kt+1$, then $Q_t \big(L^1(\mu)\big) \not \subset \mL^{\Phi}(\mu)$.}
\end{align}

\smallskip
\noindent
For $\theta < 1$, the upper bound in Theorem~\ref{thm:main} is not integrable. In this case, it provides no information on the range $Q_t\big(L^1(\mu)\big)$.

\end{itemize}

\medskip
\noindent 
\underline{Proof of \eqref{eq:log_ex_last}:}
By \eqref{eq:asymp_of_p}, 
	\begin{equation*}
		\lim\limits_{s \to 0^+} \frac{g(p(s))}{g(f^{-1}(s))} = \lim\limits_{s \to 0^+} \frac{g\Big(f^{-1}(s)\frac{p(s)}{f^{-1}(s)}\Big) }{g(f^{-1}(s))} = 
        \lim\limits_{s \to 0^+} \left(\frac{\log(f^{-1}(s)) + \log \Big( \frac{p(s)}{f^{-1}(s)}\Big)}{\log(f^{-1}(s))}\right)^{\theta} = 1 .
	\end{equation*}
	Therefore, 
		\begin{equation*}
		g(\alpha_t (u)) = g\left(f^{-1} \Big(\Big(\frac{\kappa(t)}{u}\Big)^{\frac{1}{2}}\Big)\right)  \approx g \left(p \Big(\Big(\frac{\kappa(t)}{u}\Big)^{\frac{1}{2}}\Big)\right)  \approx  \log^{\theta}\log u,
	\end{equation*}
	and, for every $c >0$ and $\lambda \geq 1$,
	\begin{align*}
		\lim\limits_{s \to 0^+}\frac{g(f^{-1}(c s^{\lambda}))}{g(f^{-1}(s))} = \lim\limits_{s \to 0^+}\frac{g(p(c s^{\lambda}))}{g(p(s))} = \lim\limits_{s \to 0^+} \Bigg(\frac{ \frac{1}{2a}\log\frac{1}{m} + \log(-\log s - \tfrac{1}{\lambda}\log c) + \log \lambda}{\frac{1}{2a}\log\frac{1}{m} + \log(-\log s)}\Bigg)^{\theta} = 1,
	\end{align*}
    i.e.\ $m=g$ satifies the condition \eqref{eq:cond_rate_v2_2}.
By Lemma \ref{lemma:lem2}(b), we get
	\begin{equation*}
		 g(\beta_t (u)) \approx \log^{\theta} \log u.
	\end{equation*}

\medskip
\noindent 
\underline{Proof of \eqref{eq:fract_orlicz_7} and \eqref{eq:fract_orlicz_72}:} As before, this follows directly from Corollary \ref{cor:cor1}(a), \eqref{eq:log_ex_last} and the asymptotic estimate $\Phi^{\prime}(u) \approx  \exp\left(c \log^{\theta} \log |u|\right)$. 

\medskip
\noindent 
\underline{Proof of \eqref{eq:fract_orlicz_8} and \eqref{eq:fract_orlicz_82}:}
We take the function $\eta$ in such a way that $\eta(r) = r \log^2 r$, for sufficiently large $r$, and recall that $\gamma_t(u)$ and $v_t(\gamma_t(u))$ are defined in Theorem \ref{prop:main3}. By following the proof of \eqref{eq:fract_orlicz_42} and \eqref{eq:fract_orlicz_4}
of Example \ref{ex:ex2}, we can easily get
\begin{align}\label{ex4:aux4}
    g(\gamma_t(u)) = \log^{\theta} \gamma_t(u) \approx \log^{\theta} \log u, \quad \text{as} \ u \to \infty.
\end{align}
By taking $\sigma(r) = r-1$ as in the previous example, for large $r$, we obtain
\[
\frac{\eta(r)}{r-\sigma(r)} =r \log^2 r,
\]
and further, for large $u$, we have
\[
\frac{\eta(\gamma_t(u))}{\gamma_t(u)-\sigma(\gamma_t(u))} =\gamma_t(u) \log^2 \gamma_t(u).
\]
This implies that
\[
v_t(\gamma_t(u)) = \exp\big(\widetilde K t \log^{\theta} \gamma_t(u)+(2\theta+2)\log \log\gamma_t(u) + \log \gamma_t(u) \big),
\]
for $u$ large enough. Hence, by \eqref{ex4:aux4}, 
\[
v_t(\gamma_t(u)) \leq  \exp\big((1+o(1))\widetilde K t \log^{\theta} \log u\big), \quad \text{as} \quad u \to \infty,
\]
whenever $\theta>1$, and
\[
v_t(\gamma_t(u)) \leq  \exp\big((1+o(1))(\widetilde K t+1) \log \log u\big), \quad \text{as} \quad u \to \infty,
\]
whenever $\theta = 1$. 

Using the asymptotic estimate $\Phi^{\prime}(u) \approx  \exp\left(c \log^{\theta} \log( |u|)\right)$, we can see that\eqref{eq:fract_orlicz_8} and \eqref{eq:fract_orlicz_82} follow from Corollary \ref{cor:cor1}(b). 
\end{example}

\subsection*{Declaration of competing interest}

We declare no conflicts of interest.

\subsection*{Data availability}

The present paper has no associated data.


\begin{thebibliography}{10}
\bibitem{bib:AbatangeloDipierroValdinoci2025}
N. Abatangelo, S. Dipierro, E. Valdinoci:\
\emph{A Gentle Invitation to the Fractional World},
UNITEXT 176,
Springer Nature Switzerland AG, Cham, 2025.
    
\bibitem{BBBOW}
K. Ball, F. Barthe, W. Bednorz, K. Oleszkiewicz, P. Wolff:\
L1-smoothing for the Ornstein–Uhlenbeck semigroup.
\emph{Mathematika} \textbf{59.1} (2013), 160-168.

\bibitem{Ban}
R. Ba\~nuelos:\ Intrinsic ultracontractivity and eigenfunction estimates for Schr\"odinger operators.
\emph{Journal of Functional Analysis} \textbf{100} (1991), 181--206.

\bibitem{BCR}
F. Barthe, P. Cattiaux, C. Roberto:\ Interpolated inequalities between exponential and
Gaussian, Orlicz hypercontractivity and isoperimetry, \emph{Revista Matematica Iberoamericana} \textbf{22(3)} (2006) 993--1067.

\bibitem{BKPP}
K. Bogdan, D. Kutek, K. Pietruska-Pałuba:
\emph{Bregman variation of semimartingales}, 
arXiv:2412.18345.

\bibitem{BSchW}
B. B\"ottcher, R.L. Schilling, J. Wang:
\emph{L\'{e}vy-type processes: construction, approximation and sample path properties}.
L\'{e}vy matters III, Springer Lecture Notes in Mathematics 2099,  Springer, Cham 2013.

\bibitem{Chen-Song1}
Z.-Q. Chen, R. Song:\ Intrinsic ultracontractivity and conditional gauge for symmetric stable processes.\emph{Journal of Functional Analysis} \textbf{150(1)} (1997), 204--239. 

\bibitem{Chen-Song2}
Z.-Q. Chen, R. Song:\ Intrinsic ultracontractivity, conditional lifetimes and conditional gauge for
symmetric stable processes on rough domains. \emph{Illinois Journal of Mathematics} \textbf{44} (2000), 138--160. 

\bibitem{ChenKimWang}
X. Chen, P. Kim, J. Wang:\ Intrinsic ultracontractivity and ground state estimates of non-local Dirichlet
forms on unbounded open sets. \emph{Communications in Mathematical Physics} \textbf{366(1)} (2019), 67--117.

\bibitem{ChenWangDom}
X. Chen, J. Wang:\ Intrinsic ultracontractivity for general L\'evy processes on bounded open sets. \emph{Illinois Journal of Mathematics} \textbf{58(4)} (2014), 1117--1144.

\bibitem{bib:ChW2015}
X. Chen, J. Wang:
Intrinsic contractivity properties of Feynman--Kac semigroups for symmetric jump processes with infinite range jumps.
\emph{Frontiers of Mathematics in China} \textbf{10} (2015), 753--776.

\bibitem{bib:ChW2016}
X. Chen, J. Wang:
Intrinsic ultracontractivity of Feynman-Kac semigroups for symmetric jump processes.
\emph{Journal of Functional Analysis} \textbf{270} (2016), 4152--4195.

\bibitem{Davies}
E.B. Davies:\ \emph{Heat Kernels and Spectral Theory}. Cambridge University Press, Cambridge 1990.

\bibitem{DGS}
E.B. Davies, L. Gross, B. Simon:\ 
Hypercontractivity:\ a bibliographic review.
In \emph{Ideas and methods in quantum and statistical physics}, 370-389. Cambridge University Press, Cambridge, 1992.

\bibitem{bib:DS}
E.B. Davies, B. Simon:\
Ultracontractivity and the heat kernel for Schr\"odinger operators and Dirichlet Laplacians.
\emph{Journal of Functional Analysis} \textbf{59} (1984), 335--395.

\bibitem{bib:DC}
M. Demuth, J.A. van Casteren:
\emph{Stochastic Spectral Theory for Self-adjoint Feller Operators. A Functional Analysis Approach}.
Birkh\"auser, Basel 2000.

\bibitem{bib:DipierroGiacominValdinoci2024}
S. Dipierro, G. Giacomin, E. Valdinoci:\
\emph{The L\'evy Flight Foraging Hypothesis in Bounded Regions:\ Subordinate Brownian Motions and High-risk/High-gain Strategies},
Memoirs of the European Mathematical Society,
European Mathematical Society (EMS Press), 2024.

\bibitem{bib:Doob}
J.L. Doob:\ Conditional Brownian motion and the boundary limits of harmonic functions, 
\emph{Bulletin de la Société Mathématique de France} \textbf{85} (1957) 431--458.

\bibitem{Eckmann}
J.P. Eckmann:\ Hypercontractivity for anharmonic oscillators, \emph{Journal of Functional Analysis} \textbf{16} (1974) 388--406.

\bibitem{Eldan-Lee}
R. Eldan, J. R. Lee:
Regularization under diffusion and anti-concentration of the information content,
\emph{Duke Mathematical Journal} \textbf{167(5)} (2018), 969-993. 

\bibitem{GLMRS}
N. Gozlan, X.-M. Li, M. Madiman, C. Roberto, P.-M. Samson: Log-Hessian and Deviation Bounds for Markov Semi-Groups, and Regularization Effect in $L^1$, \emph{Potential Analysis}, \textbf{58} (2023), 123--158.

\bibitem{GMRS}
N. Gozlan, M. Madiman, C. Roberto, P.-M. Samson: Deviation inequalities for convex functions motivated by the Talagrand conjecture, \emph{Journal of Mathematical Sciences }, \textbf{238} (2019), 453--462.

\bibitem{Gross}
L. Gross: Logarithmic Sobolev inequalities, \emph{American Journal of Mathematics} \textbf{97(4)} (1975), 1061--1083.

\bibitem{Grzywny}
T. Grzywny:\ Intrinsic ultracontractivity for L\'evy processes. \emph{Probability and Mathematical Statistics} \textbf{28} (2008), 91--106.

\bibitem{Jacob}
N. Jacob:
\emph{Pseudo-Differential Operators and Markov Processes: Markov Processes and Applications. Vol.\ I, II, III}.
Imperial College Press, London 2001--2005.

\bibitem{bib:KK}
K. Kaleta, T. Kulczycki:\ Intrinsic ultracontractivity for Schr\"odinger operators based on fractional Laplacians.
\emph{Potential Analysis} \textbf{33} (2010) 313--339.

\bibitem{bib:KKL2018}
K. Kaleta, M. Kwa\'snicki, J. L\H orinczi:\
Contractivity and ground state domination properties for non-local Schr\"odinger operators.
\emph{Journal of Spectral Theory} \textbf{8} (2018), 165--189.

\bibitem{Kaleta-Lorinczi-AoP}
K. Kaleta, J. L\H orinczi:\
Pointwise eigenfunction estimates and intrinsic ultracontractivity-type properties of Feynman--Kac semigroups for a class of L\'evy processes,
\emph{Annals of Probability}, {\bf 43} (2015), 1350--1398.
	
\bibitem{KSch}	
K. Kaleta, R.L. Schilling:\
Progressive intrinsic ultracontractivity and heat kernel estimates for non-local Schrödinger operators.
\emph{Journal of Functional Analysis} \textbf{279 (6)} (2020), Article 108606.

\bibitem{KSz}	
K. Kaleta, P. Sztonyk:\ 
Upper estimates of transition densities for stable dominated semigroups.
\emph{Journal of Evolution Equations} \textbf{13} (2013), 633--650.

\bibitem{KimSong}
P. Kim, R. Song:\ Intrinsic ultracontractivity for non-symmetric L\'evy processes. \emph{Forum Mathematicum} \textbf{21(1)} (2009), 43--66.

\bibitem{Kul}
T. Kulczycki:\ Intrinsic ultracontractivity for symmetric stable process. \emph{Bulletin of the Polish Academy of Sciences. Mathematics} \textbf{46(3)} (1998), 325--334. 

\bibitem{bib:KS}
T. Kulczycki, B. Siudeja:
Intrinsic ultracontractivity of the Feynman--Kac semigroup for relativistic stable processes.
\emph{Transactions of the American Mathematical Society} \textbf{358} (2006), 5025--5057.

\bibitem{bib:Kulczycki-Sztonyk}
T. Kulczycki, K. Sztonyk: Intrinsic ultracontractivity for Schr\"odinger semigroups based on cylindrical fractional Laplacian on the plane. \emph{Semigroup Forum}, \textbf{111} (2025), 191--226.

\bibitem{Kwa}
M. Kwa\'snicki:\ Intrinsic ultracontractivity for stable semigroups on unbounded open sets.
\emph{Potential Analysis} \textbf{31(1)} (2009), 57--77.

\bibitem{Lehec}
J. Lehec:
Regularization in $L_1$ for the Ornstein-Uhlenbeck semigroup. 
\emph{Annales de la Faculté des sciences de Toulouse:\ Mathématiques}, \textbf{25 (1)} (2016), pp. 191--204.

\bibitem{LSW}
D. Lenz, P. Stollmann, D. Wingert: Compactness of Schr\"odinger semigroups, \emph{Mathematische Nachrichten}
\textbf{283} (2010) 94--103.

\bibitem{bib:LiebSeiringer2009}
E.H. Lieb, R. Seiringer:\
\emph{The Stability of Matter in Quantum Mechanics},
Cambridge University Press, Cambridge, 2009.

\bibitem{Nelson73}
E. Nelson:
The free Markoff field.
\emph{Journal of Functional Analysis} \textbf{12} (1973), 211–-227.

\bibitem{NIST}
{\it NIST Digital Library of Mathematical Functions}.
\newblock http://dlmf.nist.gov/, Release 1.1.2 of 2021-06-15.
\newblock F.~W.~J. Olver, A.~B. {Olde Daalhuis}, D.~W. Lozier, B.~I. Schneider,
R.~F. Boisvert, C.~W. Clark, B.~R. Miller, B.~V. Saunders, H.~S. Cohl, and
M.~A. McClain, eds.

\bibitem{Rao-Ren?}
M.M.~Rao, Z.D.~Ren:
\emph{Applications Of Orlicz Spaces}
CRC Press, 2002.

\bibitem{Reed-Simon}
M.~Reed, B.~Simon:
\newblock {\em Methods of modern mathematical physics. {IV}. {A}nalysis of
	operators}.
\newblock Academic Press [Harcourt Brace Jovanovich, Publishers], New
York-London, 1978.

\bibitem{RZ}
C. Roberto, B. Zegarli\'nski:\ Hypercontractivity for Markov semigroups, 
\emph{Journal of Functional Analysis} \textbf{282} (2022), 109439.

\bibitem{bib:Sat}
K. Sato:
\emph{L\'{e}vy Processes and Infinitely Divisible Distributions}.
Cambridge University Press, Cambridge 1999.

\bibitem{bib:Sch}
R.L. Schilling: An introduction to Lévy and Feller processes. In \emph{From Lévy-type processes to parabolic
SPDEs}, 
Advanced Courses in Mathematics CRM Barcelona, 1--126. Birkh\"auser/Springer, Cham, 2016.

\bibitem{bib:SchW}
R.L. Schilling, J. Wang:\ \emph{Strong Feller continuity of Feller processes and semigroups}.
Infinite Dimensional Analysis, Quantum Probability and Related Topics \textbf{15} (2012) 1250010.

\bibitem{Schmudgen}
K. Schm\"udgen:\ \emph{Unbounded Self-adjoint Operators on Hilbert Space}, Graduate Texts in Mathematics,
vol. \textbf{265}, Springer, Dordrecht, 2012.

\bibitem{Simon}
B. Simon:\ Schr\"odinger semigroups, \emph{Bulletin of the American Mathematical Society} \textbf{7} (1982) 447--526.

\bibitem{TTT}
M. Takeda, Y. Tawara, K. Tsuchida:\ Compactness of Markov and Schrödinger semigroups: a probabilistic
approach, \emph{Osaka Journal of Mathematics} \textbf{54} (2017) 517--532.

\bibitem{Talagrand}
M. Talagrand: 
A conjecture on convolution operators, and a non-Dunford-Pettis operator on $L^1$.
\emph{Israel Journal of Mathematics} \textbf{68} (1989), no. 1, 82–88.

\bibitem{Wang}
F.-Y. Wang:\ $\Phi$--entropy inequality and application for SDEs with jumps. \emph{Journal of Mathematical Analysis and Applications} \textbf{418(2)} (2014), 861--873.

\bibitem{WangWang}
F.-Y. Wang, J. Wang:\ Isoperimetric Inequalities for Non-Local Dirichlet Forms. \emph{Potential Analysis} \textbf{53} (2020), 1225--1253.

\bibitem{WW}
 F.-Y. Wang, J.-L. Wu:\ Compactness of Schr\"odinger semigroups with unbounded below potentials, \emph{Bulletin
des Sciences Math\'ematiques} \textbf{132} (2008) 679--689.

	\end{thebibliography}
\end{document}